\documentclass[a4paper,11pt,draft]{amsart}
\usepackage{amssymb,xspace,amscd}

\title[Topology of minimal orbits]{On the topology of minimal orbits
in complex flag manifolds}

\author[A.~Altomani]{Andrea Altomani}
\address{A.\ Altomani:
Dipartimento di Matematica\\ II Universit\`a di Roma
``Tor Ver\-ga\-ta''\\ Via della Ricerca Scientifica\\ 00133 Roma
(Italy)}
\email{altomani@mat.uniroma2.it}

\author[C.~Medori]{Costantino Medori}
\address{C.\ Medori:
Dipartimento di Matematica\\ Universit\`a di Parma\\ Viale G.P.
Usberti, 53/A
\\ 43100 Parma (Italy)} \email{costantino.medori@unipr.it}

\author[M.~Nacinovich]{Mauro Nacinovich}
\address{M.\ Nacinovich:
Dipartimento di Matematica\\ II Universit\`a di Roma
``Tor Ver\-ga\-ta''\\ Via della Ricerca Scientifica\\ 00133 Roma
(Italy)}
\email{nacinovi@mat.uniroma2.it}

\date{\today}

\numberwithin{equation}{section}

\theoremstyle{plain}
\newtheorem{thm}{Theorem}[section]
\newtheorem{lem}[thm]{Lemma}

\newtheorem{prop}[thm]{Proposition}

\theoremstyle{definition}
\newtheorem{exam}[thm]{Example}
\newtheorem*{nnott}{Notation}

\newcommand{\rk}{\mathrm{rk}}

\newcommand{\bK}{\mathbf{K}}

\newcommand{\la}{\mathfrak}
\newcommand{\g}{\la{g}}

\newcommand{\gk}{{\la{k}}}
\newcommand{\gt}{{\la{t}}}

\begin{document}

\begin{abstract} We compute the Euler-Poincar\'e characteristic
of the homogeneous compact manifolds  that can be
described as minimal orbits for the action of a real form
in a complex flag manifold.
\end{abstract}

\maketitle
\tableofcontents

\section{Introduction}
A {\it complex flag manifold} is a
simply connected homogeneous compact complex manifold that is also a
projective variety. It is the quotient
$\Hat{M}=\Hat{\mathbf{G}}/\mathbf{Q}$
of a connected
complex semisimple Lie group $\Hat{\mathbf{G}}$ by a parabolic
subgroup $\mathbf{Q}$. Let a connected real form 
$\mathbf{G}$ of $\Hat{\mathbf{G}}$ act on $\Hat{M}$ by left
translations. This action 
decomposes $\Hat{M}$ 
into a
finite number of $\mathbf{G}$-orbits. Among these, there
is a unique orbit of minimal dimension, which is also
the only  one that is compact (cf. \cite{Wolf69}).
\par
In this paper we compute the Euler-Poincar\'e characteristic of the minimal
orbit $M$. This was already well
known in the two cases where either
$M=\Hat{M}$, i.e. when $\mathbf{G}$ is
transitive on $\Hat{M}$,  or $M$ is totally real, i.e. when
$\mathbf{Q}\cap\mathbf{G}$ is a real form of $\mathbf{Q}$, and,
in particular,
a real parabolic subgroup of $\mathbf{G}$. 
In these cases, indeed, explicit
cell
decompositions of $M$ were obtained by several
Authors (see e.g. \cite{CS99},
\cite{DKV83}, \cite{Koc95}).
The Euler characteristic of $M$ was also computed in \cite{MN01} for the
case where $M$ is a {\it standard} $CR$ manifold. These are indeed special
cases of minimal orbits, in which, although 
$\mathbf{Q}\cap\mathbf{G}$ is not a real form of $\mathbf{Q}$, $M$ is
diffeomorphic to a {\it real} flag manifold. \par
Our treatment of the general case, here, utilizes several 
notions developed in \cite{AMN06} for the study of the $CR$
geometry of the minimal orbits.
As in that paper, we shall use their representation
in terms of the cross-marked Satake diagrams associated to their
\emph{parabolic $CR$ algebras}.
This makes easier to deal effectively with 
their $\mathbf{G}$-equivariant fibrations, by reducing
the computation of the structure of the fibers
to combinatorics on the Satake diagrams.
\par
After observing that we may reduce to the case where $\mathbf{G}$
is simple, we show that in this case
the Euler characteristic
is different from zero, and hence positive, when $\mathbf{G}$
is compact, or of the complex type (in these cases $M$ is diffeomorphic
to a complex flag manifold), or of the real types
$\mathrm{A\,I\! I}$, $\mathrm{D\,I\! I}$ and $\mathrm{E\,I\! V}$ and
for some special real flag manifolds of the real types
$\mathrm{A\,I}$, $\mathrm{D\,I}$ and $\mathrm{E\,I}$.
We explicitly compute $\chi(M)$ when $\mathbf{Q}$ is maximal parabolic
and explain how, 
to compute $\chi(M)$ for general $M$, we may always 
reduce to that special case.\par
The paper is organized as follows. In \S\,\ref{flagsec} 
and \S\,\ref{seccrmanifold} we rehearse
the basic notions on complex flag manifolds and minimal
orbits, and prove
some results about $\mathbf{G}$-equivariant fibrations; 
in \S\,\ref{geneulero} we establish some
general criteria and tools that will be used to compute the Euler
characteristic of the minimal orbits; 
in \S\,\ref{classificazione} we prove our main results;
in \S\,\ref{esempi} we 
further illustrate our method through the discussion of
some examples; the final section \S\,\ref{appendice} is
an appendix, containing
a table that collects all the basic information 
on real semisimple Lie algebras that is required for
computing
$\chi(M)$.\par\smallskip\noindent
\begin{nnott}
Throughout this paper, a {\it hat} means that we are considering some 
{\it complexification} of the corresponding  bare object: for instance
we use $\Hat{\mathfrak{g}}$ for the complexification 
$\mathbb{C}\otimes_{\mathbb{R}}\mathfrak{g}$ of the real Lie algebra
$\mathfrak{g}$, or $\Hat{M}$ for the complex flag manifold that
contains the minimal orbit $M$. For the labels of real simple
Lie algebras and Lie groups we follow 
\cite[Table\,$\mathrm{V\! I}$, Chapter X]{Hel78}. 
For the labels of the roots and the description of the
root systems we refer to \cite{Bou68}.
\end{nnott}

\section{Complex flag manifolds}\label{flagsec}
A {\it complex flag manifold} is the quotient
$\Hat{M}=\Hat{\mathbf{G}}/\mathbf{Q}$ of a complex semisimple Lie group
$\Hat{\mathbf{G}}$ by a parabolic subgroup $\mathbf{Q}$.
We recall that
$\mathbf{Q}$ is {parabolic} in $\Hat{\mathbf{G}}$
if and only if
its Lie algebra
$\mathfrak{q}$ contains a Borel subalgebra, i.e. a maximal solvable
subalgebra, of the Lie algebra $\Hat{\mathfrak{g}}$ of
$\Hat{\mathbf{G}}$. We also note that $\Hat{\mathbf{G}}$ is necessarily
a linear group, and that $\mathbf{Q}$ is connected, contains the center
of $\Hat{\mathbf{G}}$ and equals the normalizer of $\mathfrak{q}$ in
$\Hat{\mathbf{G}}$\,:
\begin{equation}\label{qnormal}
\mathbf{Q}\,=\,\left.
\left\{g\in\Hat{\mathbf{G}}\,\right|\, \mathrm{Ad}_{\Hat{\mathfrak{g}}}(g)
(\mathfrak{q})=\mathfrak{q}\right\}\,.
\end{equation}
In particular, a different choice of 
a connected $\Hat{\mathbf{G}}'$
and of a parabolic $\mathbf{Q}'$, 
with Lie algebras $\Hat{\mathfrak{g}}'$ and
$\mathfrak{q}'$ isomorphic to $\Hat{\mathfrak{g}}$ and
$\mathfrak{q}$, yields a complex flag manifold $\Hat{M}'$ that is
complex-projectively isomorphic to $\Hat{M}$.
Thus a flag manifold $\Hat{M}$ is better described in terms of the 
pair of Lie
algebras $\Hat{\mathfrak{g}}$ and
${\mathfrak{q}}$. \par
Fix a Cartan subalgebra 
$\Hat{\mathfrak{h}}$ of
$\Hat{\mathfrak{g}}$ that is contained in $\mathfrak{q}$. 
Let $\mathcal{R}$ be the
root system with respect to  $\Hat{\mathfrak{h}}$ and denote by
$\Hat{\mathfrak{g}}^{\alpha}=\{Z\in\Hat{\mathfrak{g}}\,|\, [H,Z]=\alpha(H)Z
\;\,
\forall{H}\in\Hat{\mathfrak{h}}\}$
the root subspace of $\alpha\in\mathcal{R}$.
Then we can choose a lexicographic order 
``$\prec$'' of $\mathcal{R}$ such that
$\Hat{\mathfrak{g}}^{\alpha}\subset\mathfrak{q}$ for all positive $\alpha$.
Let $\mathcal{B}$ be the corresponding system of 
positive simple roots.
All $\alpha\in\mathcal{R}$ are linear combinations of elements 
of the basis
$\mathcal{B}$\,:
\begin{equation}\label{eq:radice}
\alpha=\sum_{\beta\in\mathcal{B}}{k_{\alpha}^{\beta}\beta}\,,\quad 
k_{\alpha}^{\beta}\in\mathbb{Z}
\end{equation}
 and we define 
the support $\mathrm{supp}_{\mathcal{B}}(\alpha)$ 
of $\alpha$ with respect to $\mathcal{B}$ as the set
of $\beta\in\mathcal{B}$ for which $k^{\beta}_{\alpha}\neq 0$.
The set
$\mathcal{Q}=\{\alpha\in\mathcal{R}\,|\,\Hat{\mathfrak{g}}^{\alpha}\subset
\mathfrak{q}\}$ is a {\it parabolic set}, i.e. 
is closed under root addition
and $\mathcal{Q}\cup\left(-\mathcal{Q}\right)=\mathcal{R}$.
Let $\Phi\subset\mathcal{B}$ be the
subset of simple roots $\alpha$ for which 
$\Hat{\mathfrak{g}}^{-\alpha}\not\subset\mathfrak{q}$. Then $\mathcal{Q}$ and
$\mathfrak{q}$ are completely determined by $\Phi$. Indeed\,:
\begin{align}
\mathcal{Q}&=\mathcal{Q}_{\Phi}:=\{\alpha\succ 0\}\cup\{\alpha\prec 0\,|\,
\mathrm{supp}_{\mathcal{B}}(\alpha)
\cap\Phi=\emptyset\}=\mathcal{Q}^r_{\Phi}\cup\mathcal{Q}^n_{\Phi}\,,
\intertext{where}
\mathcal{Q}_{\Phi}^r&=\{\alpha\in\mathcal{R}\,|\,
\mathrm{supp}_{\mathcal{B}}(\alpha)
\cap\Phi=\emptyset\}\\
\mathcal{Q}_{\Phi}^n&=\{\alpha\in\mathcal{R}\,|\,
\alpha\succ 0\;\mathrm{and}\;\mathrm{supp}_{\mathcal{B}}(\alpha)
\cap\Phi\neq\emptyset\}\,,\\
\intertext{and for the parabolic 
subalgebra $\mathfrak{q}$ we have the decomposition:}
\mathfrak{q}&=\mathfrak{q}_{\Phi}=\Hat{\mathfrak{h}}+\sum_{\alpha
\in\mathcal{Q}_{\Phi}}{\Hat{\mathfrak{g}}^{\alpha}}=\mathfrak{q}^r_{\Phi}
\oplus\mathfrak{q}^n_{\Phi}\, ,\intertext{where:}
\mathfrak{q}^n_{\Phi}&=\sum_{\alpha\in\mathcal{Q}^n_{\Phi}}
{\Hat{\mathfrak{g}}^{\alpha}}\qquad
\text{is the nilradical of $\mathfrak{q}_{\Phi}$ and }\\
\mathfrak{q}^r_{\Phi}&=\Hat{\mathfrak{h}}+\sum_{\alpha\in\mathcal{Q}^r_{\Phi}}
{\Hat{\mathfrak{g}}^{\alpha}}\qquad\text{is a reductive
complement of $\mathfrak{q}^n_{\Phi}$ in $\mathfrak{q}_{\Phi}$}.
\intertext{We also set\,:}
\Hat{\mathfrak{h}}'_{\Phi}&=\Hat{\mathfrak{h}}\cap[\mathfrak{q}_{\Phi}^r,
\mathfrak{q}_{\Phi}^r]\,,\\
\Hat{\mathfrak{h}}_{\Phi}''&=\{H\in\Hat{\mathfrak{h}}\,|\,
[H,\mathfrak{q}_{\Phi}^r]=0\}.
\intertext{Then}
\Hat{\mathfrak{h}}&=\Hat{\mathfrak{h}}'_{\Phi}\oplus
\Hat{\mathfrak{h}}_{\Phi}''
\end{align}
and $\Hat{\mathfrak{h}}_{\Phi}''$ is the center of
the reductive Lie subalgebra $\mathfrak{q}_{\Phi}^r$.\par\smallskip
All Cartan subalgebras of $\Hat{\mathfrak{g}}$ are equivalent, modulo 
inner automorphisms, and all
simple basis of a fixed  
root system $\mathcal{R}$ are equivalent for the
transpose of inner automorphisms of $\Hat{\mathfrak{g}}$ normalizing
$\Hat{\mathfrak{h}}$. Thus the correspondence $\Phi\leftrightarrow
\mathfrak{q}_{\Phi}$ is one-to-one between subsets $\Phi$ of an assigned
system $\mathcal{B}$ of simple roots of $\mathcal{R}$ and
complex parabolic Lie subalgebras
of $\Hat{\mathfrak{g}}$, modulo inner automorphisms.
In other words, the flag manifolds associated to a connected semisimple
complex Lie group with Lie algebra $\Hat{\mathfrak{g}}$ are parametrized
by the subsets $\Phi$ of a basis $\mathcal{B}$ of simple roots of its
root system $\mathcal{R}$, relative to any Cartan subalgebra
$\Hat{\mathfrak{h}}$ of  $\Hat{\mathfrak{g}}$.
\par
The choice of a Cartan subalgebra $\Hat{\mathfrak{h}}$ of
$\Hat{\mathfrak{g}}$ contained in $\mathfrak{q}$ yields a canonical
Chevalley decomposition of the parabolic subgroup $\mathbf{Q}$\,:
\begin{prop}\label{compchev}
With the notation above, we have a Chevalley decomposition\,:
\begin{equation}\label{kchev}
\mathbf{Q}=\mathbf{Q}_{\Phi}^n\ltimes\mathbf{Q}_{\Phi}^r
\end{equation}
where the unipotent radical
$\mathbf{Q}_{\Phi}^n$ is the connected and
simply connected Lie subgroup 
of $\Hat{\mathbf{G}}$ with Lie algebra $\mathfrak{q}_{\Phi}^n$,
and $\mathbf{Q}_{\Phi}^r$ is 
the reductive\footnote{According to \cite{Kn:2002} we call
reductive a linear Lie group $\mathbf{G}$, 
having finitely many connected components,
with 
a reductive Lie algebra
$\mathfrak{g}$, and such that $\mathrm{Ad}_{\Hat{\mathfrak{g}}}(\mathbf{G})
\subset\mathbf{Int}(\Hat{\mathfrak{g}})$.}
complement with Lie algebra $\mathfrak{q}^r_{\Phi}$.
The reductive $\mathbf{Q}_{\Phi}^r$ is characterized by\,:
\begin{equation}
\mathbf{Q}_{\Phi}^r
=\mathbf{Z}_{\Hat{\mathbf{G}}}(\Hat{\mathfrak{h}}_{\Phi}'')=
\{g\in\Hat{\mathbf{G}}\,|\,
\mathrm{Ad}_{\Hat{\mathfrak{g}}}(g)(H)=H\quad\forall{H}\in
\Hat{\mathfrak{h}}_{\Phi}''\}\,.
\end{equation}
Moreover 
$\mathbf{Q}_{\Phi}^r$ is a subgroup of finite index in 
$\mathbf{N}_{\Hat{\mathbf{G}}}(\mathfrak{q}^r_{\Phi})=
\{g\in\Hat{\mathbf{G}}\,|\,\mathrm{Ad}_{\Hat{\mathfrak{g}}}(
\mathfrak{q}^r_{\Phi})=\mathfrak{q}^r_{\Phi}\}$ and
$\mathbf{Q}\cap\mathbf{N}_{\Hat{\mathbf{G}}}(\mathfrak{q}^r_{\Phi})=
\mathbf{Q}_{\Phi}^r$.
\end{prop}
\begin{proof}
A complex parabolic subgroup can also be considered as a \emph{real}
parabolic subgroup. The Chevalley decomposition  \eqref{kchev}
reduces then to the Langlands decomposition 
$\mathbf{Q}=\mathbf{MAN}$, with $\mathbf{N}=\mathbf{Q}^n_{\Phi}$ and
$\mathbf{MA}=\mathbf{Q}^r_{\Phi}$. Thus our statement reduces to
\cite[Proposition 7.82(a)]{Kn:2002}. \par
Next we note that $\mathfrak{q}^r_{\Phi}$ is the centralizer of
$\Hat{\mathfrak{h}}_{\Phi}''$ in $\Hat{\mathfrak{g}}$ and is its own
normalizer. This yields the inclusion
$\mathbf{Q}_{\Phi}^r\subset
\mathbf{N}_{\Hat{\mathbf{G}}}(\mathfrak{q}^r_{\Phi})$.
Since $\mathbf{N}_{\Hat{\mathbf{G}}}(\mathfrak{q}^r_{\Phi})$ is
semi-algebraic, it has finitely many connected components.
Thus its intersection with $\mathbf{Q}^n_{\Phi}$ is discrete and
finite, and thus trivial because $\mathbf{Q}^n_{\Phi}$ is connected,
simply connected and unipotent. 
\end{proof}
\section{The structure of minimal 
$\mathbf{G}$-orbits}\label{seccrmanifold}
Let
$\Hat{M}=\Hat{\mathbf{G}}/\mathbf{Q}$
be a flag manifold for the transitive action of the 
connected semisimple complex linear Lie group $\Hat{\mathbf{G}}$, and
$\mathbf{G}$ a connected real form of  $\Hat{\mathbf{G}}$.
Note that $\mathbf{G}$ is semi-algebraic, being a topological connected
component of an algebraic group.
We know from \cite{Wolf69} that there are finitely many $\mathbf{G}$-orbits.
Fix any orbit $M$ and a point $x\in{M}$. We can assume that
$\mathbf{Q}\subset\Hat{\mathbf{G}}$ is the stabilizer of $x$ 
for the action of $\Hat{\mathbf{G}}$ in $\Hat{M}$. 
We keep the notation of \S\ref{flagsec}, and we also set
$\mathbf{G}_+=\mathbf{Q}\cap\mathbf{G}$ for the stabilizer of $x$ in
$\mathbf{G}$, so that $M\simeq\mathbf{G}/\mathbf{G}_+$. Let
$\mathfrak{g}\subset\Hat{\mathfrak{g}}$ be the Lie algebra of
$\mathbf{G}$ and
$\mathfrak{g}_+=\mathfrak{q}\cap\mathfrak{g}$ the Lie algebra of 
$\mathbf{G}_+$. 
\par
We summarize the results of
\cite[p.491]{AMN06} by stating the following\,:
\begin{prop} \label{ccaa}
With the notation above\,:
$\mathfrak{g}_+$ contains a Cartan subalgebra $\mathfrak{h}$ of
$\mathfrak{g}$. If $\mathfrak{h}$ is any Cartan subalgebra of $\mathfrak{g}$
contained in $\mathfrak{g}_+$, there are a Cartan involution
$\vartheta:\mathfrak{g}\to\mathfrak{g}$ and a decomposition
\begin{equation}\label{algisdec}
\mathfrak{g}_+=\mathfrak{n}\oplus\mathfrak{w}=\mathfrak{n}\oplus
\mathfrak{l}\oplus\mathfrak{z}
\end{equation} 
such that\,:
\begin{itemize}
\item
$\mathfrak{n}$ is the nilpotent ideal 
of $\mathfrak{g}_+$, consisting of the elements $X\in\mathfrak{g}_+$
for which $\mathrm{ad}_{\mathfrak{g}}(X):\mathfrak{g}\to\mathfrak{g}$
is nilpotent;
\item
$\mathfrak{w}=\mathfrak{l}\oplus\mathfrak{z}$ is reductive;
\item
$\mathfrak{z}\subset\mathfrak{h}$
is the center of $\mathfrak{w}$ and
$\mathfrak{l}=[\mathfrak{w},\mathfrak{w}]$ its semisimple ideal;
\item $\mathfrak{h}$,
$\mathfrak{n}$, $\mathfrak{z}$ and $\mathfrak{l}$ are
invariant under the Cartan involution $\vartheta$ of $\mathfrak{g}$.\qed
\end{itemize}
\end{prop}
We have the following\,:
\begin{prop} Keep the notation introduced above.
The isotropy subgroup $\mathbf{G}_+$ is the closed real
semi-algebraic
subgroup of $\mathbf{G}$\,:
\begin{equation}
\mathbf{G}_+=\mathbf{N}_{\mathbf{G}}(\mathfrak{q}_{\Phi})=\{
g\in\mathbf{G}\,|\, \mathrm{Ad}_{\Hat{\mathfrak{g}}}(g)(\mathfrak{q}_{\Phi})=
\mathfrak{q}_{\Phi}\}\,.\end{equation}
The isotropy subgroup $\mathbf{G}_+$ admits a Chevalley decomposition
\begin{equation}\label{alfisdec}
\mathbf{G}_+=\mathbf{W}\rtimes\mathbf{N}
\end{equation}
where\,:
\begin{itemize}
\item
$\mathbf{N}$ is a unipotent, closed, connected, and simply connected
subgroup with Lie algebra $\mathfrak{n}$;
\item
$\mathbf{W}$ is a reductive Lie subgroup,  with Lie algebra $\mathfrak{w}$,
and is the centralizer of $\mathfrak{z}$ in $\mathbf{G}$\,:
\begin{equation}\label{trequattro}
\mathbf{W}=\mathbf{Z}_{\mathbf{G}}(\mathfrak{z})=
\{g\in\mathbf{G}\,|\, \mathrm{Ad}_{\mathfrak{g}}(g)(H)=
H\quad\forall{H}\in\mathfrak{z}\}\, .
\end{equation}
\end{itemize}
\end{prop} 
\begin{proof}
Let $g\in\mathbf{G}_+$. Then $\mathrm{Ad}_{\mathfrak{g}}(g)(\mathfrak{w})$
is a reductive complement of $\mathfrak{n}$ in $\mathfrak{g}_+$.
Since all reductive complements of $\mathfrak{n}$ are conjugated
by an inner automorphism from 
$\mathrm{Ad}_{\mathfrak{g}_+}(\mathbf{N})$, we can find a $g_n\in\mathbf{N}$
such that $\mathrm{Ad}_{\mathfrak{g}_+}(g_n^{-1}g)(\mathfrak{w})=
\mathfrak{w}$. Consider the element $g_r=g_n^{-1}g$. 
We have\,:
\begin{align*} 
\mathrm{Ad}_{\mathfrak{g}}(g_r)(\mathfrak{w})&=\mathfrak{w},&
\mathrm{Ad}_{\Hat{\mathfrak{g}}}(g_r)(\mathfrak{q}_{\Phi})&=\mathfrak{q}_{\Phi},\\
\mathrm{Ad}_{\Hat{\mathfrak{g}}}(g_r)(\mathfrak{q}_{\Phi}^n)&=
\mathfrak{q}_{\Phi}^n ,&
\mathrm{Ad}_{\Hat{\mathfrak{g}}}(g_r)(\bar{\mathfrak{q}}_{\Phi})&=
\bar{\mathfrak{q}}_{\Phi},
\end{align*}
because $g_r\in\mathbf{Q}\cap\bar{\mathbf{Q}}$.
We consider the parabolic subalgebra of $\Hat{\mathfrak{g}}$ defined
by\,:
\begin{equation*}
\mathfrak{q}_{\Phi'}=\mathfrak{q}_{\Phi}^n
\oplus\left(\mathfrak{q}_{\Phi}^r\cap\bar{\mathfrak{q}}_{\Phi}\right)=
\mathfrak{q}_{\Phi}^n+
\left(\mathfrak{q}_{\Phi}\cap\bar{\mathfrak{q}}_{\Phi}\right)\,.
\end{equation*}
It has the property that $\mathfrak{q}^r_{\Phi'}=
\bar{\mathfrak{q}}^r_{\Phi'}$ is the complexification of
$\mathfrak{w}$. Clearly  
$\mathrm{Ad}_{\Hat{\mathfrak{g}}}(g_r)(\mathfrak{q}_{\Phi'})=
\mathfrak{q}_{\Phi'}$
and  
$\mathrm{Ad}_{\Hat{\mathfrak{g}}}(g_r)(\mathfrak{q}_{\Phi'}^r)=
\mathfrak{q}_{\Phi'}^r$. By Proposition \ref{compchev},
$g_r\in\mathbf{Z}_{\Hat{\mathbf{G}}}(\Hat{\mathfrak{h}}_{\Phi'}'')$.
The statement follows because $g_r\in\mathbf{G}$ and
$\Hat{\mathfrak{h}}_{\Phi'}''$ is the complexification of $\mathfrak{z}$.
\end{proof}
Among the $\mathbf{G}$-orbits in $\Hat{M}$ there is one,
and only one, say $M$, that is closed, and that we shall call henceforth
{\it the minimal orbit}.
Fix a point $x\in{M}$. We can assume that
$\mathbf{Q}$ is the stabilizer of $x$ in $\Hat{\mathbf{G}}$. 
Then the orbit $M$ 
is completely
determined by the datum of the
real Lie algebra $\mathfrak{g}$ of $\mathbf{G}$ and of the complex Lie
subalgebra $\mathfrak{q}$ of $\Hat{\mathfrak{g}}$ corresponding to 
$\mathbf{Q}$. In \cite{AMN06}
we called the pair $(\mathfrak{g},\mathfrak{q})$, consisting of
the real Lie algebra $\mathfrak{g}$ 
and of the parabolic complex Lie subalgebra $\mathfrak{q}$
of its complexification $\Hat{\mathfrak{g}}$, a {\it parabolic 
minimal $CR$ algebra}.
This is a special instance of the notion of $CR$ algebra that was introduced
in \cite{MN05} (for the general orbits and their corresponding {\it parabolic
$CR$ algebras}, we refer the reader to \cite{AMN06b}).\par
We recall that $(\mathfrak{g},\mathfrak{q})$ is {\it effective}
if $\mathfrak{g}_+$ does not contain any ideal of $\mathfrak{g}$.
We remark that this means that the action of $\mathbf{G}$ on $M$ is
almost effective.\par
Moreover, we have (see \cite[p.490]{AMN06})\,:
\begin{prop}\label{prodemme}
Let $M$ be the minimal orbit associated to the pair
$(\mathfrak{g},\mathfrak{q})$. 
If $\mathfrak{g}=\bigoplus_{i=1}^m{\mathfrak{g}_i}$ is
the decomposition of $\mathfrak{g}$ into the direct sum of its simple
ideals, then 
\begin{enumerate}
\item
$\mathfrak{q}_i=\mathfrak{q}\cap\Hat{\mathfrak{g}}_i$
is parabolic in $\Hat{\mathfrak{g}}_i$;
\item 
$\mathfrak{q}=\bigoplus_{i=1}^m{\mathfrak{q}_i}$; 
\item $M=M_1\times\cdots\times{M}_m$, where $M_i$ is a minimal orbit
associated to the pair $(\mathfrak{g}_i,\mathfrak{q}_i)$;
\item $(\mathfrak{g},\mathfrak{q})$ is effective if and only if
all $(\mathfrak{g}_i,\mathfrak{q}_i)$ are effective, i.e. if
$\mathfrak{q}_i\neq\Hat{\mathfrak{g}}_i$, for all $i=1,\hdots,m$.\qed
\end{enumerate} 
 \end{prop}
We showed in \cite[Proposition 5.5]{AMN06} that 
$\mathfrak{g}_+=\mathfrak{g}\cap\mathfrak{q}$ contains a 
maximally noncompact Cartan subalgebra $\mathfrak{h}$ of $\mathfrak{g}$.
Fix such a maximally noncompact 
Cartan subalgebra $\mathfrak{h}\subset\mathfrak{g}_+$ 
of $\mathfrak{g}$, and, accordingly, a Cartan involution
$\vartheta$ and a decomposition \eqref{algisdec} as in
Proposition \ref{ccaa}.
\par
Let 
\begin{equation}\label{cartdecomp}
\mathfrak{g}=
\mathfrak{k}\oplus\mathfrak{p}
\end{equation}
 be the
Cartan decomposition defined by $\vartheta$. Then
$\mathfrak{h}=\mathfrak{h}^+\oplus\mathfrak{h}^-$, 
with $\mathfrak{h}^+=\mathfrak{h}\cap\mathfrak{k}$ and
$\mathfrak{h}^-=\mathfrak{h}\cap\mathfrak{p}$. Moreover,
$\mathfrak{k}$ is the Lie algebra of a maximal compact subgroup
$\mathbf{K}$ of $\mathbf{G}$. The group $\mathbf{K}$ is connected
and semi-algebraic; hence the isotropy subgroup
$\mathbf{K}_+=\mathbf{K}\cap\mathbf{Q}$ has finitely many
connected components and thus, by \cite{MR0037311}, 
acts transitively on the minimal orbit $M$, so that\,:
\begin{equation}\label{cappa}
M=
{\mathbf{K}}/{\displaystyle \mathbf{K}_+}\,.
\end{equation}
\par
Let $\mathcal{R}$ be the root system 
of $\hat{\mathfrak{g}}$ with respect to $\hat{\mathfrak{h}}$.
By duality, the conjugation in $\Hat{\mathfrak{g}}$ defined by the
real form $\mathfrak{g}$ defines an involution $\alpha\rightarrow\bar{\alpha}$
in the root system  $\mathcal{R}$. A root $\alpha$ is
{\it real} when $\bar\alpha=\alpha$, 
{\it imaginary} when $\bar\alpha=-\alpha$,
{\it complex} when $\bar\alpha\neq\pm\alpha$. The condition that
$\mathfrak{h}$ is maximally noncompact is equivalent to the fact that all
imaginary roots $\alpha$ are compact, i.e. that 
$\Hat{\mathfrak{g}}^{\alpha}\subset\Hat{\mathfrak{k}}=
\mathbb{C}\otimes\mathfrak{k}$. We indicate by $\mathcal{R}_{\bullet}$
the set of imaginary roots.
\par
We also showed (see \cite[Proposition 6.2]{AMN06}) that, by
choosing a suitable lexicographic order in $\mathcal{R}$, we have,
with the notation of \S \ref{flagsec}\,:
\begin{enumerate}
\item $\mathcal{R}^+=\{\alpha\succ{0}\}\subset\mathcal{Q}$\,,
\item $\bar\alpha\succ{0}$ for all complex $\alpha$ in $\mathcal{R}^+$.
\end{enumerate}
Let $\mathcal{B}$ be the system of simple roots in $\mathcal{R}^+$.
The involution $\alpha\rightarrow\bar\alpha$ defines an involution
$\alpha\rightarrow\varepsilon(\alpha)$ on 
$\mathcal{B}\setminus\mathcal{R}_{\bullet}$,
with the property that $\bar\alpha=\varepsilon(\alpha)+\sum_{\beta\in
\mathcal{B}\cap\mathcal{R}_{\bullet}}{t_{\alpha}^{\beta}\beta}$. It is
described on the corresponding Satake diagrams (cf. \cite{Ara62})
by joining by a curved arrow all pairs of distinct simple roots
$(\alpha,\varepsilon(\alpha))$. 
\par\smallskip
Let $\Phi\subset\mathcal{B}$ 
and $\mathfrak{q}=\mathfrak{q}_{\Phi}$
be as in \S \ref{flagsec}.
Then the Satake diagram 
$\mathcal{S}$ of $\mathfrak{g}$,
with cross-marks corresponding to the roots in $\Phi$, yields a
complete graphic description of the minimal orbit $M$ 
(see \cite[\S 6]{AMN06}). We call the pair $(\mathcal{S},\Phi)$ the
{\it cross-marked Satake diagram} associated to $M$, or,
equivalently, to the
pair $(\mathfrak{g},\mathfrak{q})$.
\par\smallskip
An inclusion 
$\mathbf{Q}_{\Phi}\subset\mathbf{Q}_{\Phi'}$ 
of parabolic subgroups of $\Hat{\mathbf{G}}$ defines a natural
$\Hat{\mathbf{G}}$-equivariant fibration $\left(\Hat{\mathbf{G}}/
\mathbf{Q}_{\Phi}\right)\to
\left(\Hat{\mathbf{G}}/
\mathbf{Q}_{\Phi'}\right)$, yielding by restriction a
$\mathbf{G}$-equivariant fibration $M\to{M}'$ of the corresponding
minimal orbits. In the following proposition
we describe these 
$\mathbf{G}$-equivariant fibrations
in terms of the associated cross-marked Satake diagrams\,:
\begin{prop}\label{fibrazione}
Let $M$ and $M'$ be minimal orbits for the same $\mathbf{G}$,
associated to the pairs
$(\mathfrak{g},\mathfrak{q}_{\Phi})$
and $(\mathfrak{g},\mathfrak{q}_{\Phi'})$, 
respectively,
with $\Phi'\subset\Phi$. 
Let $\mathcal{E}$ be the set of all roots 
$\alpha\in\mathcal{B}$
with 
$\left(\{\alpha\}\cup\mathrm{supp}_{\mathcal{B}}
(\bar\alpha)\right)\cap\Phi'\neq
\emptyset$. Consider the Satake diagram
$\mathcal{S}''$ obtained from the Satake diagram $\mathcal{S}$ of 
$\mathfrak{g}$
by erasing all nodes corresponding to 
the set $\mathcal{E}$ and all lines and arrows issued from them.\par
Then the $\mathbf{G}$-equivariant fibration $M\to{M}'$
has connected fibers,
that are minimal orbits $M''$, corresponding to the 
cross-marked Satake diagram $(\mathcal{S}'',\Phi'')$, where
$\Phi''=\Phi\setminus\mathcal{E}$.\par
\end{prop}
\begin{proof}
Let $\mathbf{H}=\mathbf{Z}_{\mathbf{G}}(\mathfrak{h})=
\{h\in\mathbf{G}\,|\, \mathrm{Ad}_{\mathfrak{g}}(h)(H)=H\,,\;
\forall H\in
\mathfrak{h}\}$ be the Cartan subgroup of
$\mathbf{G}$ corresponding to $\mathfrak{h}$. We have
$\mathrm{Ad}_{\Hat{\mathfrak{g}}}(h)(\mathfrak{q}_{\Phi})=\mathfrak{q}_{\Phi}$ 
and
$\mathrm{Ad}_{\Hat{\mathfrak{g}}}(h)(\mathfrak{q}_{\Phi'})=
\mathfrak{q}_{\Phi'}$ for all
$h\in \mathbf{H}$. Hence 
$\mathbf{H}\subset\mathbf{G}_+\subset\mathbf{G}'_+$,
where $\mathbf{G}_+=\mathbf{G}\cap\mathbf{Q}_{\Phi}$ and
$\mathbf{G}_+'=\mathbf{G}\cap\mathbf{Q}_{\Phi'}$.
We decompose $\mathbf{G}'_+=\mathbf{W}'\rtimes\mathbf{N}'$
according to \eqref{alfisdec}, with a $\mathbf{W}'$ that
satisfies \eqref{trequattro}. Then $\mathbf{H}\subset\mathbf{W}'$ and,
since $\mathfrak{h}$ is maximally
noncompact in $\mathfrak{g}_+'$, it is also maximally noncompact in
$\mathfrak{w}'=\mathrm{Lie}(\mathbf{W}')$. Thus,
by \cite[Proposition 7.90]{Kn:2002}, 
all connected components of $\mathbf{W}'$, and hence also of
$\mathbf{G}'_+$, 
intersect $\mathbf{H}$
and, a fortiori, $\mathbf{G}_+$. 
Therefore the fiber $\mathbf{G}'_+/\mathbf{G}_+$ is connected.
\par
The fact that $M''$ is the minimal orbit associated to $(\mathcal{S}'',
\Phi'')$ is the contents of \cite[Proposition 7.3]{AMN06}.
\end{proof}
In the following two lemmata we give 
sufficient conditions,
in terms of cross-marked Satake diagrams, in order that
two minimal orbits be diffeomorphic.
\begin{lem}\label{modellodebole}
We keep the notation introduced above. Let\,:
\begin{align}
\Pi=\Phi\cup\{\alpha\in\mathcal{B}\setminus\mathcal{R}_{\bullet}\,|\,
{\mathrm{supp}}(\bar\alpha)\cap\Phi\neq\emptyset\}\,,
\end{align}
and let $M^*$ be the minimal orbit corresponding to 
$(\mathfrak{g},\mathfrak{q}_{\Pi})$. Then
the canonical $\mathbf{G}$-equivariant map 
$M^*\to{M}$ is a diffeomorphism.
\end{lem}
\begin{proof} By Proposition \ref{fibrazione}, $M^*\to{M}$ 
is a $\mathbf{G}$-equivariant fibration whose fiber reduces to
a point, and hence a diffeomorphism. \end{proof}
From this Lemma we obtain\,:
\begin{lem}\label{lemdiffeom} We keep the notation introduced above.
Let $M_1$, $M_2$ be minimal orbits associated to pairs 
$(\mathfrak{g},\mathfrak{q}_{\Phi_1})$, $(\mathfrak{g},\mathfrak{q}_{\Phi_2})$,
respectively, for the same semisimple real Lie algebra $\mathfrak{g}$,
and with suitable $\Phi_1,\Phi_2\subset\mathcal{B}$.
Let 
\begin{align*}
\Pi_1=\Phi_1\cup\{\alpha\in\mathcal{B}\,|\,
{\mathrm{supp}}(\bar\alpha)\cap\Phi_1\neq
\emptyset\}\,,\\
\Pi_2=\Phi_2\cup\{\alpha\in\mathcal{B}\,|\,
{\mathrm{supp}}(\bar\alpha)\cap\Phi_2\neq
\emptyset\}\,.\end{align*}
If $\Pi_1=\Pi_2$, 
then there is a $\mathbf{G}$-equivariant diffeomorphism $M_1\to{M}_2$. 
\end{lem}
\begin{proof} Indeed, by Lemma \ref{modellodebole}, we have 
a chain of $\mathbf{G}$-equivariant diffeomorphisms
$\begin{CD} M_1@<\sim<< M_1^*@=M_2^*@>\sim>> M_2\,. \end{CD}$
\end{proof}
We also have\,:
\begin{prop}\label{essesecondo}
We keep the notation introduced above. 
By erasing all nodes corresponding to roots in $\Pi$ and all
lines and arrows issuing from them, we obtain a new Satake diagram
$\mathcal{S}''_{\Phi}$, that is the Satake diagram of a Levi
subalgebra $\mathfrak{l}$ of $\mathfrak{g}_+=\mathfrak{q}\cap\mathfrak{g}$.
Then
$\mathcal{R}''_{\Phi}=\mathcal{Q}^r_{\Phi}\cap\bar{\mathcal{Q}}^r_{\Phi}$
is the root system of the complexification $\Hat{\mathfrak{l}}$ of
$\mathfrak{l}$ with respect to 
its Cartan subalgebra $\Hat{\mathfrak{h}}\cap\Hat{\mathfrak{l}}$, that is
the complexification
of the maximally noncompact Cartan subalgebra
$\mathfrak{h}\cap\mathfrak{l}$~of~$\mathfrak{l}$.
\end{prop}
\begin{proof}
Since $\mathfrak{q}_{\Pi}\cap\mathfrak{g}=\mathfrak{q}_{\Phi}\cap\mathfrak{g}
=\mathfrak{g}_+$, we can as well assume that $\Phi=\Pi$. The
intersection 
$\mathfrak{w}=\mathfrak{q}^r\cap\bar{\mathfrak{q}}^r\cap\mathfrak{g}$
is a reductive complement of the nilradical of $\mathfrak{g}_+$
and its semisimple ideal $\mathfrak{l}=[\mathfrak{w},\mathfrak{w}]$ is
a Levi subalgebra of $\mathfrak{g}_+$. The associated root system
of $\mathfrak{q}^r\cap\bar{\mathfrak{q}}^r$, with respect to
$\Hat{\mathfrak{h}}$, 
and of $\Hat{\mathfrak{l}}$ with respect to its Cartan subalgebra
$\Hat{\mathfrak{h}}\cap\Hat{\mathfrak{l}}$, is 
$\mathcal{Q}^r_{\Pi}\cap\bar{\mathcal{Q}}^r_{\Pi}$.
We observe that $\bar\alpha\in\mathcal{Q}^r_{\Pi}$ for all
simple $\alpha\in\mathcal{B}\setminus\Pi$. Hence, for
$\alpha\in\mathcal{Q}^r_{\Pi}$, also $\bar\alpha\in\mathcal{Q}^r_{\Pi}$,
because $\mathrm{supp}_{\mathcal{B}}(\bar\alpha)\subset\bigcup_{\beta\in
\mathrm{supp}_{\mathcal{B}}(\alpha)}\mathrm{supp}_{\mathcal{B}}(\bar\beta)$
and hence, when $\mathrm{supp}_{\mathcal{B}}(\alpha)\cap\Pi=\emptyset$,
also $\mathrm{supp}_{\mathcal{B}}(\bar\alpha)\cap\Pi=\emptyset$.
This shows that $\mathcal{Q}^r_{\Pi}=\bar{\mathcal{Q}}^r_{\Pi}$ and
that
$\mathrm{supp}_{\mathcal{B}}(\alpha)\subset\mathcal{B}
\setminus\Pi$ for all 
$\alpha\in\mathcal{Q}^r_{\Pi}$. Since $\mathcal{B}\setminus\Pi\subset
\mathcal{Q}^r_{\Pi}$, we proved that $\mathcal{B}\setminus\Pi$ is a
system of simple roots for $\mathcal{R}''_{\Phi}=\mathcal{Q}^r_{\Pi}=
\mathcal{Q}^r_{\Pi}\cap\bar{\mathcal{Q}}^r_{\Pi}$.
Since the nodes of $\mathcal{S}''_{\Phi}$ are exactly those corresponding
to the simple roots in $\mathcal{B}\setminus\Pi$,
this proves our contention.
\end{proof} 
From Lemma \ref{lemdiffeom}, we obtain in particular\,:
\begin{prop} If $\mathfrak{g}$ is a simple Lie algebra 
of the complex type,
then every minimal orbit $M$ of $\mathbf{G}$ 
is diffeomorphic to a complex flag manifold.
\end{prop}
\begin{proof} The Satake diagram of $\mathfrak{g}$ consists of two disjoint
connected graphs, whose nodes correspond to two
sets of simple roots, each root of one 
set being strongly orthogonal to all roots of the other,
$\mathcal{B}'=\{\alpha'_1,\hdots,\alpha'_{\ell}\}$ and
$\mathcal{B}''=\{\alpha''_1,\hdots,\alpha''_{\ell}\}$, with curved arrows
joining $\alpha'_j$ to $\alpha''_j$. 
Let $J\subset
\{1,\hdots,\ell\}$ be the set of indices for which either $\alpha'_j$ or
$\alpha_j''$ are cross-marked, i.e. belongs to $\Phi\subset\mathcal{B}=
\mathcal{B}'\cup\mathcal{B}''$. By Lemma \ref{lemdiffeom},
our $M$ is diffeomorphic to the $M'$ corresponding to the parabolic
$\mathfrak{q}_{\Phi'}$ with
$\Phi'=\{\alpha'_j\,|\, j\in{J}\}$.
By \cite[Theorem 10.2]{AMN06},
$M'$ is complex and, hence, a complex flag manifold.
\end{proof}

\section{Euler characteristic of minimal orbits}\label{geneulero}
Let $M=\mathbf{K}/\mathbf{K}_+$ be a homogeneous space for the transitive
action of a compact connected Lie group $\mathbf{K}$.
It is known
(see e.g.
\cite[p.182]{MR0336651})
that 
its Euler characteristic $\chi(M)$ is nonnegative. Moreover,
it is positive exactly when the rank of the isotropy subgroup
$\mathbf{K}_+$ equals the rank of
$\mathbf{K}$. 
In this case the identity component $\mathbf{K}_+^0$ of
the isotropy $\mathbf{K}_+$ contains the center of $\mathbf{K}$
and hence $\tilde{M}=\mathbf{K}/\mathbf{K}_+^0$ is the universal
covering of $M$.
Indeed, we can reduce to the case of a semisimple $\mathbf{K}$
and thus assume that $\mathbf{K}$ is simply connected.
The number of sheets of $\tilde{M}\to{M}$
equals then the order $|\pi_1(M)|$ of the fundamental group of
$M$. 
By using e.g.
\cite[Ch.VII, Theorem 3.13]{MR1122592}, we obtain that\,:
\begin{align}
\chi(\tilde{M})&=\frac{|\mathbf{W}(\mathbf{K})|}{
|\mathbf{W}(\mathbf{K}_+^0)|}\,,\\
\chi(M)&=\frac{|\mathbf{W}(\mathbf{K})|}{|\mathbf{W}(\mathbf{K}_+^0)|\cdot
|\pi_1(M)|}\,.
\end{align}
\medskip\par
We have the following\,:
\begin{prop}\label{prop-rank}
Let $M=\mathbf{G}/\mathbf{G}_+=
\mathbf{K}/\mathbf{K}_+$, as in \eqref{cappa},
be a minimal orbit. Then $\mathbf{K}_+=\mathbf{K}\cap\mathbf{G}_+$
is a maximal compact
subgroup of $\mathbf{G}_+$, contained in
the maximal
compact subgroup
$\mathbf{K}$ of $\mathbf{G}$, and
the following are equivalent:
\begin{enumerate}
\item $\chi(M)>0$;
\item
$\rk(\mathbf{K})=\rk(\mathbf{K}_+)$,
i.e. $\mathbf{K}_+$ contains a maximal torus of $\mathbf{K}$;
\item
$\g_+$ contains a maximally compact Cartan subalgebra of
$\g$.
\end{enumerate}
\end{prop}
\begin{proof}
The equivalence $(1)\Longleftrightarrow(2)$ is the contents of \cite{Wan1949}, 
thus we need only to prove that $(2)\Longleftrightarrow(3)$.
We also observe that $\mathbf{K}_+$ is a maximal compact subgroup
of $\mathbf{G}_+$ because of \eqref{cappa}.
\par
Let $\gk$ and $\gk_+$ be the Lie algebras of $\bK$ and 
$\bK_+$,
respectively.
Assume that $\gk_+$ contains a maximal torus $\gt$ of $\gk$.
Take a maximal Abelian subalgebra
$\mathfrak{a}$ of $\mathfrak{g}_+$, consisting of
$\mathrm{ad}_{\mathfrak{g}}$-semisimple elements, and
with $\mathfrak{a}\supset\gt$. We claim that
$\mathfrak{a}$ is a Cartan subalgebra of $\g_+$
and therefore also of $\g$, and clearly it will also be
maximally compact in $\mathfrak{g}$.
Indeed, since $\g_+$ contains a
Cartan subalgebra of $\g$,
all Cartan subalgebras of $\mathfrak{g}_+$ are also
Cartan subalgebras of $\mathfrak{g}$. 
Since $\mathfrak{g}_+$ is $\mathrm{ad}_{\mathfrak{g}}$-splittable
(see \cite[Proposition 5.4]{AMN06}), 
its Cartan subalgebras are its maximal Abelian
Lie subalgebras consisting of 
$\mathrm{ad}_{\mathfrak{g}}$-semisimple elements
(cf. e.g. \cite[Chap.VII,\S5, Prop.6]{Bou75}).
\par
Vice versa, if $\mathfrak{a}$ is a 
maximally compact Cartan subalgebra of $\g$ contained in
$\g_+$, then $\mathfrak{a}\cap\gk
=\mathfrak{a}\cap\gk_+$ is a maximal torus of $\gk$ and $\gk_+$.
Thus $\mathbf{K}$ and $\mathbf{K}_+$ have the same rank. 
\end{proof}
In the following
we shall keep the notation of \S\ref{seccrmanifold}.
In particular, we fix a maximally noncompact Cartan subalgebra
$\mathfrak{h}$ of $\mathfrak{g}$ contained in $\mathfrak{g}_+$,
standard with respect to the Cartan decomposition
\eqref{cartdecomp} associated to the Cartan involution $\vartheta$
of Proposition \ref{ccaa}.
To express the
equivalent conditions $(1)$, $(2)$ and $(3)$ of Proposition \ref{prop-rank}
in terms of the description in \S\ref{seccrmanifold}, 
we need to rehearse first the construction of the Cartan subalgebras
of a real semisimple Lie algebra from \cite{Kos55},
\cite{Su59} 
and \cite{Kn:2002}\,:
\begin{lem}\label{standardcartan} 
With the notation above: every Cartan subalgebra
of $\mathfrak{g}$ is equivalent, modulo an inner automorphism,
to a Cartan subalgebra $\mathfrak{a}$ 
which is standard with respect to the triple $(\mathfrak{k},
\mathfrak{p},\mathfrak{h}^-)$. This means that
$\mathfrak{a}$ has 
noncompact part $\mathfrak{a}^-\subset\mathfrak{h}^-$ and
compact part $\mathfrak{a}^+\subset\mathfrak{k}$.\par
All standard Cartan subalgebras  $\mathfrak{a}$ are obtained in
the following way:
\begin{enumerate}
\item fix a system $\alpha_1,\hdots,\alpha_r$ of strongly orthogonal
real roots in $\mathcal{R}$; 
\item fix $X_{\pm\alpha_i}\in\Hat{\mathfrak{g}}^{\pm\alpha_i}\cap\mathfrak{g}$
with $[X_{-\alpha_i},X_{\alpha_i}]=H_{\alpha_i}$,
$[H_{\alpha_i},X_{\pm\alpha_i}]=\pm{2}X_{\pm\alpha_i}$, for $i=1,\hdots,r$;
\item let 
$\mathbf{d}=\mathbf{d}_{\alpha_1}\circ\cdots\circ\mathbf{d}_{\alpha_r}$,
where
$\mathbf{d}_{\alpha_i}={\mathrm{Ad}}_{\Hat{\mathfrak{g}}}
\left(\exp(i\pi(X_{-\alpha_i}
-X_{\alpha_i})/4)\right)$, for $i=1,\hdots,r$, 
($\mathbf{d}$ is the Cayley transform with respect to
$\alpha_1,\hdots,\alpha_r$);
\item set $\mathfrak{a}=\mathbf{d}(\Hat{\mathfrak{h}})\cap\mathfrak{g}$.
\qed \end{enumerate}
\end{lem}
\begin{nnott}
For a real semisimple Lie algebra $\mathfrak{g}$,
with associated Satake's diagram $\mathcal{S}$,  
we shall denote by $\nu=\nu(\mathfrak{g})=
\nu(\mathcal{S})$ the maximum number of strongly orthogonal real roots
in $\mathcal{R}$. \end{nnott}
From Lemma \ref{standardcartan} we deduce the criterion\,:
\begin{prop}\label{criterio}
Let $M$ be the minimal orbit corresponding to the pair
$(\mathfrak{g},\mathfrak{q}_{\Phi})$. Let $\mathfrak{l}$ be a Levi subalgebra
of $\mathfrak{g}_+$. 
Then $\chi(M)>0$ if and only if one of the following equivalent conditions
is satisfied\,:
\begin{align} 
&\left\{\begin{aligned}
&\text{$\mathcal{Q}^r_{\Phi}$
contains a maximal system of  strongly orthogonal}\\
&\text{real roots of
$\mathcal{R}$;}\end{aligned}\right.
\label{primocriterio}
\\
&\nu(\mathfrak{l})=\nu(\mathfrak{g}).\label{nucriterio}
\end{align}
\end{prop}  
\begin{proof} 
The Cartan subalgebras  of $\mathfrak{g}$ contained in
$\mathfrak{g}_+$ are conjugated, modulo inner automorphisms of
$\mathfrak{g}_+$, to standard Cartan subalgebras  that are contained in
$\mathfrak{w}=\mathfrak{q}^r\cap\mathfrak{g}$. Decompose 
the reductive real Lie algebra $\mathfrak{w}$ as
$\mathfrak{w}=\mathfrak{l}\oplus\mathfrak{z}$, where
$\mathfrak{z}$ is the center of $\mathfrak{w}$ and
$\mathfrak{l}=[\mathfrak{w},\mathfrak{w}]$ its semisimple ideal,
that is a Levi subalgebra of $\mathfrak{g}_+$.
We have $\mathfrak{z}\subset\mathfrak{h}$ and
$\mathfrak{h}=\mathfrak{z}\oplus\left(\mathfrak{h}\cap\mathfrak{l}\right)$.
Thus a maximally compact Cartan subalgebra of $\mathfrak{g}_+$ will
be conjugate to one of the form $\mathfrak{z}\oplus\mathfrak{e}$, 
with $\mathfrak{e}$ a maximally compact Cartan subalgebra of
$\mathfrak{l}$. By Lemma \ref{standardcartan}, these are obtained
via a Cayley transform $\mathbf{d}=\mathbf{d}_{\alpha_1}\circ\cdots\circ
\mathbf{d}_{\alpha_r}$ for a system of strongly orthogonal real roots
$\alpha_1,\hdots,\alpha_r$ in $\mathcal{Q}^r_{\Phi}$. 
Hence the statement follows.
\end{proof}

\section{Classification of the minimal orbits with $\chi(M)>0$}
\label{classificazione}
Throughout this section, we shall 
consistently employ the notation of the
previous sections. In particular, $\mathfrak{l}$ will always denote
a Levi subalgebra of $\mathfrak{g}_+$, 
$\mathfrak{k}$ the compact Lie subalgebra in
the decomposition \eqref{cartdecomp}. We
set $\mathfrak{k}_+^s$ for the maximal compact subalgebra
$\mathfrak{k}\cap\mathfrak{l}$ of 
$\mathfrak{l}$.
\par\smallskip
By using the result of Proposition \ref{fibrazione}, the computation
of $\chi(M)$ for a minimal orbit $M$ can be reduced to the
the case where, for the associated $CR$ algebra $(\mathfrak{g},\mathfrak{q})$,
the real Lie algebra ${\mathfrak{g}}$ is simple and
the parabolic $\mathfrak{q}$ is maximal, i.e. 
$\Phi=\{\alpha\}$ for some $\alpha\in\mathcal{B}$. 
Thus we begin
by considering this special case\,:
\begin{thm}\label{singolo} 
Let $M$ be the minimal orbit associated to the 
effective pair
$(\mathfrak{g},\mathfrak{q}_{\{\alpha\}})$, with ${\mathfrak{g}}$ simple 
and a maximal parabolic $\mathfrak{q}_{\{\alpha\}}\subset\Hat{\mathfrak{g}}$, 
for $\alpha\in\mathcal{B}$. \par
Then $\chi(M)>0$ 
if and only if either one of the following conditions holds: 
\begin{itemize}
\item
$\mathfrak{g}$ is either
of the complex
type, or compact, or of the real 
non split types $\mathrm{A\,I\!{I}}$,
$\mathrm{D\,I\!{I}}$, $\mathrm{E\,I\!{V}}$
and $\alpha$ is any root in $\mathcal{B}$;
\item $\mathfrak{g}$ is of the real types
$\mathrm{A\,I}$, $\mathrm{D\,I}$, $\mathrm{E\,I}$,
and $\alpha\in\mathcal{B}$ is chosen as in the table below.
\end{itemize}
Here we list all pairs
of real noncompact $\mathfrak{g}$ and $\alpha\in\mathcal{B}$
for which $\chi(M)>0$, also computing
$\chi(M)$ in the different cases.
\[{\scriptsize
\begin{array}{|c|c|c|c|c|c|}\hline
&&&&&\\[-5pt]
\mathrm{type}&\mathfrak{g}&\alpha&\mathrm{condition}&\chi(\tilde{M})
&\chi(M)\\
&&&&&\\[-5pt]\hline
&&&&&\\[-5pt]
\mathrm{A\,I}&\mathfrak{sl}(n,\mathbb{R})&\alpha_i&i\cdot(n-i)\in{2}\mathbb{Z}&
2\binom{[n/2]}{[i/2]}
&\binom{[n/2]}{[i/2]}
\\
&&&&&\\[-5pt] \hline
&&&&&\\[-5pt]
&&\alpha_{2i-1}&1\leq i\leq{n}&
2i\binom{n}{i}&2i\binom{n}{i}\\
\mathrm{A\,I\!{I}}&
\mathfrak{sl}(n,\mathbb{H})&&&&\\[-5pt]
\cline{3-6}
&&&&&\\[-5pt]
&&\alpha_{2i}&1\leq{i}<{n}&
\binom{n}{i}&\binom{n}{i}\\
&&&&&\\[-5pt] \hline
&&&&&\\[-5pt]
\mathrm{D\,I}&
\begin{gathered}
\mathfrak{so}(p,2n-p)\\
n\geq 4\\
2\leq p\leq{n}
\end{gathered}
&\alpha_1&p+1\in{2}\mathbb{Z}&
4&2\\ 
&&&&&\\[-5pt] \hline
&&&&&\\[-5pt]
&
&\alpha_1&
&
2&2
\\ 
&&&&&\\[-5pt]
\cline{3-6}
&\mathfrak{so}(1,2n-1) &&&&\\[-5pt]
\mathrm{D\,I\!{I}}&&
\alpha_i&
2\leq i\leq n-2
&
2^i\binom{n-1}{i-1}&2^i\binom{n-1}{i-1}\\
&n\geq 4&&&&
\\ \cline{3-6}
&&&&&\\[-5pt]
&
&\alpha_{i}&
n-1\leq i\leq{n}&
2^{n-1}&2^{n-1}\\
&&&&&
\\[-5pt] \hline
&&&&&\\[-5pt]
\mathrm{E\,I}&
\mathfrak{e}_{\mathrm{I}}
&\alpha_{i}&
i\in\{1,6\}&
6&3
\\[4pt]
\hline
\end{array}}
\]

\[{\scriptsize
\begin{array}{|c|c|c|c|c|c|}\hline
&&&&&\\[-5pt]
\mathrm{type}&\mathfrak{g}&\alpha&\mathrm{condition}&\chi(\tilde{M})
&\chi(M)\\
&&&&&\\[-5pt]\hline
&&&&&\\[-5pt]
&&\alpha_{i}&
i\in\{1,6\}&
3&3\\
&&&&&
\\[-5pt] \cline{3-6}
&&&&&\\[-5pt]
\mathrm{E\,I\!{V}}&
\mathfrak{e}_{\mathrm{I\!{V}}}
&
\alpha_{i}&
i=2,3,5&
192&192\\
&&&&&
\\[-5pt] \cline{3-6}
&&&&&\\[-5pt]
&&\alpha_{4}&
&
144&144\\[4pt]
\hline
\end{array}}
\]
\end{thm}
\par\medskip

\begin{proof} 
When $\mathfrak{g}$ is either of the complex type, or compact, or
of the real types $\mathrm{A\,I\!{I}}$, $\mathrm{D\,I\!{I}}$,
$\mathrm{E\,I\!{V}}$, we have 
$\nu(\mathfrak{g})=\nu(\mathcal{S})=0$ and thus
the necessary and sufficient condition of Proposition \ref{criterio}
to have $\chi(M)>0$
is trivially satisfied.
Moreover, we know from \cite[Theorem 8.6]{AMN06} that
$M$ is simply connected, and therefore $\tilde{M}=M$ and
$\chi(\tilde{M})=\chi(M)$.
\par
We shall discuss the remaining cases mainly
by comparing $\nu(\mathcal{S})$
with $\nu(\mathcal{S}''_{\{\alpha\}})$ for the different types
of $\mathfrak{g}$.\par

\subsection*{$\mathrm{[A\,I]}$} Here $\mathfrak{g}=\mathfrak{sl}(n,\mathbb{R})$
and $\alpha=\alpha_i$ with $1\leq i<n$. Then 
$\mathfrak{l}\simeq\mathfrak{sl}(i,\mathbb{R})\oplus
\mathfrak{sl}(n-i,\mathbb{R})$; hence
$\nu(\mathcal{S}''_{\{\alpha_i\}})=[i/2]+[(n-i)/2]$ and thus
$[n/2]=[i/2]+[(n-i)/2]$, i.e. $i(n-i)\in{2}\mathbb{Z}$,
is the necessary and sufficient condition in order that
$\nu(\mathcal{S})=\nu(\mathcal{S}''_{\{\alpha\}})$.
We have $\mathfrak{k}\simeq\mathfrak{so}(n)$ and
$\mathfrak{k}^s_+\simeq \mathfrak{so}(i)\oplus\mathfrak{so}(n-i)$.
Thus, when $\chi(M)>0$ we have\,:
$$\chi(\tilde{M})=\frac{2^{[(n+1)/2]-1}[n/2]!}{2^{[(i+1)/2]-1}[i/2]!\cdot
2^{[(n-i+1)/2]-1}[(n-i)/2]!}=2\binom{[n/2]}{[i/2]}\,.$$
We have $\pi_1(M)\simeq\mathbb{Z}_2$ (see e.g. \cite{Wig98}),
and hence
$\chi(M)=\binom{[n/2]}{[i/2]}$.
\subsection*{$\mathrm{[A\,I\!{I}]}$} Here $\mathfrak{g}=\mathfrak{sl}(n,
\mathbb{H})$, and
$\nu(\mathcal{S})=0$ yields $\chi(M)>0$ for any choice of $\alpha$.
If $\alpha=\alpha_{2i-1}$, for $1\leq i\leq n$, then
$\mathfrak{l}\simeq \mathfrak{sl}(i-1,\mathbb{H})\oplus
\mathfrak{sl}(n-i,\mathbb{H})$. Hence $\mathfrak{k}\simeq\mathfrak{sp}(n)$
and $\mathfrak{k}_+^s\simeq\mathfrak{sp}(i-1)\oplus\mathfrak{sp}(n-i)$.
Thus $$\chi(M)=\chi(\tilde{M})
=\frac{2^n n!}{2^{i-1}(i-1)!\cdot2^{n-i}(n-i)!}=2i\binom{n}{i}\,.$$
If $\alpha=\alpha_{2i}$ with $1\leq i<n$, then 
$\mathfrak{l}\simeq\mathfrak{sl}(i,\mathbb{H})\oplus
\mathfrak{sl}(n-i,\mathbb{H})$. Thus 
$\mathfrak{k}_+^s\simeq\mathfrak{sp}(i)\oplus\mathfrak{sp}(n-i)$
and
$$\chi(M)=\chi(\tilde{M})=\frac{2^n n!}{(2^{i}(i)!)(2^{n-i}(n-i)!)}
=\binom{n}{i}\,.$$
\subsection*{$\mathrm{[A\,I\! I\! I,\; A\,I\! V]}$} 
We have $\mathfrak{g}\simeq\mathfrak{su}(p,q)$ with
$p\leq q$, $p+q=\ell+1$. Then $\nu(\mathcal{S})=p>0$.
We obtain, for the different choices of $\alpha=\alpha_i\in\mathcal{B}$\,:
$${\scriptsize
\mathfrak{l}\simeq\left\{\begin{aligned}
&\mathfrak{sl}(i,\mathbb{C})\oplus\mathfrak{su}(p-i,q-i)
&&\text{if}\quad 1\leq i\leq p-1&&\Rightarrow 
\nu(\mathcal{S}''_{\alpha})=p-i<p\\
&\mathfrak{sl}(p,\mathbb{C})\oplus \mathfrak{su}(i-p)\oplus
\mathfrak{su}(q-i) &&\text{if}\quad p\leq i\leq q
&&\Rightarrow \nu(\mathcal{S}''_{\alpha})=0<p\\
&\mathfrak{sl}(p+q-i,\mathbb{C})\oplus\mathfrak{su}(i-q,i-p)
&&\text{if}\quad q+1\leq i\leq\ell&&\Rightarrow 
\nu(\mathcal{S}''_{\alpha})=i-q<p
\end{aligned}\right.}
$$
showing that
condition 
(\ref{nucriterio})
is never
fulfilled.
\subsection*{$\mathrm{[B\,I\,, \; B\,I\!{I}]}$} 
We have $\mathfrak{g}\simeq\mathfrak{so}(p,2n+1-p)$, with
$1\leq p\leq{n}$. We have $\nu(\mathcal{S})=p$, while,
for $\alpha=\alpha_i$ we obtain\,:
$${\scriptsize
\mathfrak{l}\simeq\left\{\begin{aligned}
&\mathfrak{sl}(i,\mathbb{R})\oplus\mathfrak{so}(p-i,2n+1-p-i)
&&\text{if}\quad 1\leq i<p&&\Rightarrow \nu(\mathcal{S}''_{\alpha})=
[i/2]+p-i<p\\
&\mathfrak{sl}(p,\mathbb{R})\oplus \mathfrak{so}(2n+1-2p)
&&\text{if}\quad i=p
&&\Rightarrow \nu(\mathcal{S}''_{\alpha})=[p/2]<p\\
&\mathfrak{sl}(p,\mathbb{R})\oplus\mathfrak{su}(i-p)
\oplus\mathfrak{so}(2n+1-2i)
&&\text{if}\quad p<i\leq{n}
&&\Rightarrow \nu(\mathcal{S}''_{\alpha})=[p/2]<p
\end{aligned}\right.}
$$
yielding also in this case $\chi(M)=0$.
\subsection*{$\mathrm{[C\,I]}$} Here $\mathfrak{g}=\mathfrak{sp}(2n,\mathbb{R})$.
We have $\nu(\mathcal{S})=n$, while, for $\alpha=\alpha_i$, with 
$1\leq i\leq{n}$, we obtain that $\mathfrak{l}\simeq \mathfrak{sl}(i,\mathbb{R})
\oplus\mathfrak{sp}(2(n-i),\mathbb{R})$. Hence
$\nu(\mathcal{S}''_{\alpha})=[i/2]+(n-i)<n$ implies that
$\chi(M)=0$. 
\subsection*{$\mathrm{[C\,II]}$} 
We have $\mathfrak{g}\simeq\mathfrak{sp}(p,\ell-p)$, with $1\leq 2p\leq \ell$.
In this case all positive real roots are pairwise strongly orthogonal.
There is a positive real root, namely 
$\beta=\alpha_1+2(\alpha_2+\cdots+\alpha_{\ell-1})+\alpha_{\ell}$,
with $\mathrm{supp}_{\mathcal{B}}(\beta)=\mathcal{B}$. Therefore we have
$\beta\in\mathcal{Q}^n_{\{\alpha\}}$,
hence $\nu(\mathcal{S}''_{\{\alpha\}})<\nu(\mathcal{S})$ and
$\chi(M)=0$, for all possible choices of
$\alpha\in\mathcal{B}$.
\subsection*{$\mathrm{[D\,I]}$}
We have $\mathfrak{g}\simeq\mathfrak{so}(p,2n-p)$,
with $2\leq p\leq{n}$ and $\nu(\mathcal{S})=2\left[\frac{p}{2}\right]$.
Because of the symmetry of $\mathcal{S}$, the minimal orbits corresponding
to $\Phi=\{\alpha_{n-1}\}$ and to $\Phi=\{\alpha_{n}\}$ are diffeomorphic.
Thus we can assume in the following that $i\neq{n-1}$.
We obtain\,:
$${\scriptsize
\mathfrak{l}\simeq\left\{\begin{aligned}
&\mathfrak{sl}(i,\mathbb{R})\oplus\mathfrak{so}(p-i,2n-p-i)
&&\text{if}\; 1\leq i\leq {p}
&\Rightarrow& \nu(\mathcal{S}''_{\alpha})=
\left[\frac{i}{2}\right]+2\left[\frac{p-i}{2}\right]\\
&\mathfrak{sl}(p,\mathbb{R})\oplus \mathfrak{su}(i-p)\oplus\mathfrak{so}(2n-2i)
&&\text{if}\; p<i\leq{n},\; i\neq{n-1}
&&\Rightarrow \nu(\mathcal{S}''_{\alpha})=\left[\frac{p}{2}\right]<2
\left[\frac{p}{2}\right]\,.
\end{aligned}\right.}
$$
The equation $[i/2]+2[(p-i)/2]=2[p/2]$, for integral $i$ with
$1\leq i\leq p$ is solvable if and only if $p$ is odd, and in this
case we also need to have $i=1$. Thus $\chi(M)>0$ if and only
if $p=2h+1$ is odd and $\alpha=\alpha_1$. Then 
$\mathfrak{k}=\mathfrak{so}(2h+1)\oplus\mathfrak{so}(2n-2h-1)$ and
$\mathfrak{k}^s_+\simeq \mathfrak{so}(2h)\oplus\mathfrak{so}(2n-2h-2)$.
Hence in this case\,:
$$\chi(\tilde{M})=\frac{2^hh!\cdot2^{n-h-1}(n-h-1)!}
{2^{h-1}h!\cdot 2^{n-h-2}(n-h-1)!}
=4\,.$$
We have
$\pi_1(M)\simeq\mathbb{Z}_2$ (see e.g. \cite{Wig98}), and 
hence $\chi(M)=2$.
\subsection*{$\mathrm{[D\,I\! I]}$}
We have $\mathfrak{g}\simeq\mathfrak{so}(1,2n-1)$, with $n\geq 4$, 
and $\mathcal{R}$ does not contain any real root, so that 
condition  
(\ref{nucriterio}) 
is trivially fulfilled.
We have $\mathfrak{k}\simeq\mathfrak{so}(2n-1)$.
When $\alpha=\alpha_1$, we have 
$\mathfrak{k}^s_+\simeq\mathfrak{so}(2n-2)$. Thus
$$\chi(M)=\chi(\tilde{M})=\frac{2^{n-1}(n-1)!}{2^{n-2}(n-1)!}=2\,.$$
If $\alpha=\alpha_i$ with $2\leq i\leq n-2$, we obtain
$\mathfrak{k}^s_+\simeq \mathfrak{su}(i-1)\oplus\mathfrak{so}(2n-2i)$
and\,:
$$\chi(M)=\chi(\tilde{M})=\frac{2^{n-1}(n-1)!}{(i-1)!2^{n-i-1}(n-i)!}=
2^i\binom{n-1}{i-1}\,.$$
When $\alpha\in\{\alpha_{n-1},\alpha_n\}$, we obtain
$\mathfrak{k}^s_+\simeq\mathfrak{su}(n-1)$ and therefore\,:
$$\chi(M)=\chi(\tilde{M})=\frac{2^{n-1}(n-1)!}{(n-1)!}=2^{n-1}\,.$$  
\subsection*{$\mathrm{[D\,I\! I\! I]}$} 
We have $\mathfrak{g}\simeq\mathfrak{so}^*(2n)$.
In this case all positive real roots are pairwise strongly orthogonal,
and therefore form a maximal system of strongly orthogonal real roots.
Since $\beta=\alpha_1+\alpha_{n-1}+\alpha_n+2\sum_{i=2}^{n-2}{\alpha_i}$
is a positive real root with $\mathrm{supp}(\beta)=\mathcal{B}$, 
condition 
(\ref{nucriterio}) is not
fulfilled for any choice of $\alpha\in\mathcal{B}$.
\subsection*{The exceptional Lie algebras}
We shall discuss the case of the noncompact real forms of the
exceptional Lie algebras by comparing $\nu=\nu(\mathfrak{g})=
\nu(\mathcal{S})$ with $\nu''=\nu(\mathfrak{l})=
\nu(\mathcal{S}''_{\{\alpha\}})$. Since the proceeding is
straightforward, we limit ourselves to list 
the Levi subalgebra
$\mathfrak{l}$ of $\mathfrak{g}_+$ and the corresponding value
of $\nu''$,
for each different choice of $\alpha\in\mathcal{B}$.
\begin{equation*}\scriptsize
\begin{array}{|c|c|c|c|c|c|}\hline
&&&&&\\[-5pt]
\mathrm{type}&\mathfrak{g}&\nu&\alpha&\mathfrak{l}&\nu''\\ 
&&&&&\\[-5pt]
\hline
&&&&&\\[-5pt]
&&&\alpha_1,\alpha_6&\mathfrak{so}(5,5)&4\\ &&&&&\\[-5pt] \cline{4-6} 
&&&&&\\[-5pt]
\mathrm{E\,I} &\mathfrak{e}_{\mathrm{I}} & 4
&\alpha_2&\mathfrak{sl}(6,\mathbb{R})&3\\ &&&&&\\[-5pt] \cline{4-6} 
&&&&&\\[-5pt]
&&&\alpha_3,\alpha_5&\mathfrak{sl}(2,\mathbb{R})
\oplus\mathfrak{sl}(5,\mathbb{R})
&3\\  &&&&&\\[-5pt] \cline{4-6} &&&&&\\[-5pt]
&&&\alpha_4&\mathfrak{sl}(3,\mathbb{R})\oplus\mathfrak{sl}(2,\mathbb{R})\oplus
\mathfrak{sl}(3,\mathbb{R})&3\\ 
 &&&&&\\[-5pt] \hline &&&&&\\[-5pt]
&&&
\alpha_1,\alpha_6&\mathfrak{so}(3,5)&2\\
&&&&&\\[-5pt] \cline{4-6} &&&&&\\[-5pt]
\mathrm{E\,I\!{I}}&
\mathfrak{e}_{\mathrm{I\! I}}&
4&
\alpha_2&\mathfrak{su}(3,3)&3\\ &&&&&\\[-5pt] \cline{4-6} &&&&&\\[-5pt]
&&&\alpha_3,\alpha_5&
\mathfrak{sl}(2,\mathbb{C})\oplus\mathfrak{sl}(3,\mathbb{R})
&1\\ &&&&&\\[-5pt] \cline{4-6} &&&&&\\[-5pt]
&&&\alpha_4&\mathfrak{sl}(3,\mathbb{C})\oplus
\mathfrak{sl}(2,\mathbb{R})&1\\ &&&&&\\[-5pt] \hline &&&&&\\[-5pt]
&&&
\alpha_1,\alpha_6&\mathfrak{so}(1,7)&0\\ &&&&&\\[-5pt] \cline{4-6} &&&&&\\[-5pt]
\mathrm{E\,I\!{I}\!{I}}&
\mathfrak{e}_{\mathrm{I\! I\! I}}
& 2&\alpha_2&\mathfrak{su}(1,5)&1\\ &&&&&\\[-5pt] \cline{4-6} &&&&&\\[-5pt]
&&&\alpha_3,\alpha_5&\mathfrak{su}(3)
&0\\ &&&&&\\[-5pt] \cline{4-6} &&&&&\\[-5pt]
&&&\alpha_4&\mathfrak{su}(2)\oplus\mathfrak{su}(2)
&0\\ &&&&&\\[-5pt] \hline
&&&&&\\[-5pt]
&&&\alpha_1,\alpha_6&\mathfrak{so}(8)&0\\ &&&&&\\[-5pt] \cline{4-6} 
&&&&&\\[-5pt]
\mathrm{E\,I\!{V}} &
\mathfrak{e}_{\mathrm{I\! V}}&
 0&\alpha_2,
\alpha_3,\alpha_5&\mathfrak{su}(4)
&0\\ &&&&&\\[-5pt] \cline{4-6} &&&&&\\[-5pt]
&&&\alpha_4&\qquad
\mathfrak{su}(2)\oplus\mathfrak{su}(2)\oplus\mathfrak{su}(2)
\quad&0 \\ &&&&&\\[-5pt] \hline 
&&&&&\\[-5pt]
&&&
\alpha_1&\mathfrak{so}(6,6)&6\\ &&&&&\\[-5pt] \cline{4-6} &&&&&\\[-5pt]
&&&\alpha_2&\mathfrak{sl}(7,\mathbb{R})&3\\ &&&&&\\[-5pt] \cline{4-6} 
&&&&&\\[-5pt]
&&&\alpha_3&\mathfrak{sl}(2,\mathbb{R})\oplus\mathfrak{sl}(6,\mathbb{R})&4\\ 
&&&&&\\[-5pt] \cline{4-6} &&&&&\\[-5pt]
\mathrm{E\,V}&
\mathfrak{e}_{\mathrm{V}}&
 7&\alpha_4&\mathfrak{sl}(3,\mathbb{R})\oplus\mathfrak{sl}(2,\mathbb{R})
\oplus\mathfrak{sl}(4,\mathbb{R})&4\\ &&&&&\\[-5pt] \cline{4-6} &&&&&\\[-5pt]
&&&\alpha_5&\mathfrak{sl}(5,\mathbb{R})\oplus\mathfrak{sl}(3,\mathbb{R})
&3\\ &&&&&\\[-5pt] \cline{4-6} &&&&&\\[-5pt]
&&&\alpha_6&\mathfrak{so}(5,5)\oplus\mathfrak{sl}(2,\mathbb{R})
&5\\ &&&&&\\[-5pt] \cline{4-6} &&&&&\\[-5pt]
&&&\alpha_7&\mathfrak{e}_I
&4\\[4pt]
\hline 
\end{array}
\end{equation*}
\begin{equation*}\scriptsize
\begin{array}{|c|c|c|c|c|c|}\hline
&&&&&\\[-5pt]
\mathrm{type}&\mathfrak{g}&\nu&\alpha&\mathfrak{l}&\nu''\\ \hline
&&&&&\\[-6pt]
&&&
\alpha_1&\mathfrak{so}^*(12)& 3\\&&&&&\\[-6pt] \cline{4-6} &&&&&\\[-6pt] &&&
\alpha_2&\mathfrak{sl}(3,\mathbb{R})\oplus\mathfrak{sl}(2,\mathbb{H})& 1\\ 
&&&&&\\[-6pt] \cline{4-6} &&&&&\\[-6pt] &&&
\alpha_3&\mathfrak{sl}(2,\mathbb{R})\oplus\mathfrak{sl}(3,\mathbb{H})& 1\\ 
&&&&&\\[-6pt] \cline{4-6} &&&&&\\[-6pt] 
\mathrm{E\,V\!{I}}&
\mathfrak{e}_{\mathrm{V\! I}}&
 4&\alpha_4
&\mathfrak{sl}(3,\mathbb{R})\oplus\mathfrak{su}(2)
\oplus\mathfrak{sl}(2,\mathbb{H})& 1\\ &&&&&\\[-6pt] \cline{4-6} &&&&&\\[-6pt] &&&
\alpha_5&\mathfrak{sl}(3,\mathbb{R})\oplus\mathfrak{su}(2)\oplus\mathfrak{su}(2)
&1\\ &&&&&\\[-6pt] \cline{4-6} &&&&&\\[-6pt] &&&
\alpha_6&\mathfrak{so}(3,7)\oplus\mathfrak{su}(2)
&2\\ &&&&&\\[-6pt] \cline{4-6} &&&&&\\[-6pt] &&&
\alpha_7&\mathfrak{so}(3,7)
&2\\ &&&&&\\[-6pt] \hline &&&&&\\[-6pt]
&&&&&\\[-6pt]
&&&
\alpha_1&\mathfrak{so}(2,10)& 2\\&&&&&\\[-6pt] \cline{4-6} &&&&&\\[-6pt] 
&&&
\alpha_2,
\alpha_3,\alpha_5&
\mathfrak{su}(4)\oplus\mathfrak{sl}(2,\mathbb{R})& 1\\ &&&&&\\[-6pt] 
\cline{4-6} &&&&&\\[-6pt] 
\mathrm{E\,V\!{I}\!{I}}&
\mathfrak{e}_{\mathrm{V\! I\! I}}&
 3&
\alpha_4&\mathfrak{su}(2)
\oplus\mathfrak{su}(2)\oplus\mathfrak{su}(2)\oplus
\mathfrak{sl}(2,\mathbb{R})
& 1\\ &&&&&\\[-6pt] \cline{4-6} &&&&&\\[-6pt] &&&
\alpha_6&\mathfrak{so}(1,9)\oplus\mathfrak{sl}(2,\mathbb{R})
&1\\ &&&&&\\[-6pt] \cline{4-6} &&&&&\\[-6pt] &&&
\alpha_7&\mathfrak{e}_{IV}
&0\\ &&&&&\\[-6pt] \hline
&&&&&\\[-6pt]
&&&
\alpha_1&\mathfrak{so}(7,7)&6\\&&&&&\\[-6pt] \cline{4-6} &&&&&\\[-6pt] &&&
\alpha_2&\mathfrak{sl}(8,\mathbb{R})&4\\ &&&&&\\[-6pt] \cline{4-6} &&&&&\\[-6pt] &&&
\alpha_3&\mathfrak{sl}(2,\mathbb{R})\oplus\mathfrak{sl}(7,\mathbb{R})&4\\ 
&&&&&\\[-6pt] \cline{4-6} &&&&&\\[-6pt] 
\mathrm{E\,V\!{I}\!{I}\!{I}}&
\mathfrak{e}_{\mathrm{V\! I\! I\! I}}&
 8&
\alpha_4&\mathfrak{sl}(3,\mathbb{R})\oplus\mathfrak{sl}(2,\mathbb{R})
\oplus\mathfrak{sl}(5,\mathbb{R})&4\\ &&&&&\\[-6pt] \cline{4-6} &&&&&\\[-6pt] &&&
\alpha_5&\mathfrak{sl}(5,\mathbb{R})\oplus\mathfrak{sl}(4,\mathbb{R})
&4\\ &&&&&\\[-6pt] \cline{4-6} &&&&&\\[-6pt] &&&
\alpha_6&\mathfrak{so}(5,5)\oplus\mathfrak{sl}(3,\mathbb{R})
&5\\ &&&&&\\[-6pt] \cline{4-6} &&&&&\\[-6pt] &&&
\alpha_7&\mathfrak{e}_I\oplus\mathfrak{sl}(2,\mathbb{R})
&5\\ &&&&&\\[-6pt] \cline{4-6} &&&&&\\[-6pt] &&&
\alpha_8&\mathfrak{e}_{V}
&7\\ &&&&&\\[-6pt] \hline 
&&&&&\\[-6pt]
&&&
\alpha_1&\mathfrak{so}(3,11)&2\\&&&&&\\[-6pt] \cline{4-6} &&&&&\\[-6pt] &&&
\alpha_2,
\alpha_3,\alpha_5 
&\mathfrak{su}(4)
\mathfrak{sl}(3,\mathbb{R})&1\\ &&&&&\\[-6pt] \cline{4-6} &&&&&\\[-6pt]
\mathrm{E\,I\!{X}}&
\mathfrak{e}_{\mathrm{I\! X}}&
 4&
\alpha_4&\mathfrak{su}(2)\oplus\mathfrak{su}(2)\oplus\mathfrak{su}(2)\oplus
\mathfrak{sl}(3,\mathbb{R})
&1\\ &&&&&\\[-6pt] \cline{4-6} &&&&&\\[-6pt] &&&
\alpha_6&\mathfrak{so}(1,9)\oplus\mathfrak{sl}(3,\mathbb{R})
&1\\ &&&&&\\[-6pt] \cline{4-6} &&&&&\\[-6pt] &&&
\alpha_7&\mathfrak{e}_{IV}\oplus\mathfrak{sl}(2,\mathbb{R})
&1\\ &&&&&\\[-6pt] \cline{4-6} &&&&&\\[-6pt] &&&
\alpha_8&\mathfrak{e}_{VII}
&3
\\&&&&&\\[-6pt] \hline &&&&&\\[-6pt]
&&&
\alpha_1&\mathfrak{sp}(6,\mathbb{R})&3\\&&&&&\\[-6pt] \cline{4-6} &&&&&
\\\mathrm{F\,I}&
\mathfrak{f}_{\mathrm{I}}&
 4&
\alpha_2,
\alpha_3
&
\mathfrak{sl}(2,\mathbb{R})\oplus\mathfrak{sl}(3,\mathbb{R})&2
\\ &&&&&\\[-6pt] \cline{4-6} &&&&&\\[-6pt]
&&& \alpha_4&\mathfrak{so}(3,4)
&3
\\ &&&&&\\[-6pt] \hline &&&&&\\[-6pt] 
&&&
\alpha_1&\mathfrak{so}(5)&0\\&&&&&\\[-6pt] \cline{4-6} &&&&&\\[-6pt]
\mathrm{F\,I\!{I}}&
\mathfrak{f}_{\mathrm{I\! I}}&
 1&
\alpha_2&\mathfrak{su}(2)\oplus\mathfrak{su}(2) &0\\&&&&&\\[-6pt] \cline{4-6} &&&&&\\[-6pt]
&&&\alpha_3
&
\mathfrak{su}(3)&0\\ &&&&&\\[-6pt] \cline{4-6} &&&&&\\[-6pt]
&&& \alpha_4&\mathfrak{so}(7)
&0
\\ &&&&&\\[-6pt] \hline &&&&&\\[-6pt]
\mathrm{G\,I}&
\mathfrak{g}_{\mathrm{I}}&
 2&
\alpha_1,
\alpha_2&\mathfrak{sl}(2,\mathbb{R})&1\\[3pt]
\hline
\end{array}
\end{equation*}
\par\bigskip
Looking up to the list, we see that $\nu=\nu''$ if, and only if,
either\,:
\begin{enumerate}
\item[$(i)$] $\mathfrak{g}$ is of type $\mathrm{E\,I}$ and
$\alpha\in\{\alpha_1,\alpha_6\}$, or 
\item[$(ii)$]
$\mathfrak{g}$ is of type $\mathrm{E\,I\!{V}}$ and $\alpha$ is
any element of $\mathcal{B}$.\end{enumerate}
In case $(i)$, $\mathfrak{k}^s_+\simeq \mathfrak{so}(5)\oplus
\mathfrak{so}(5)$ and hence $|\mathbf{W}(\mathbf{K}_+)|=64=2^6$.
We have $\mathfrak{k}=\mathfrak{sp}(4)$ and hence
$|\mathbf{W}(\mathbf{K})|=384=2^44!$. Thus $\chi(\tilde{M})=384/64=6$.
Finally, since 
$\pi_1(M)\simeq\mathbb{Z}_2$ (see e.g. \cite{Wig98}), 
the manifold
$\tilde{M}$ is a two-fold covering
of $M$, and we obtain that $\chi(M)=3$. 
\par
In case $(ii)$, we have $\mathfrak{k}=
\mathfrak{f}_{\mathrm{I\! I\! I}}$ (the
compact form of the complex simple Lie algebra of type $\mathrm{F}_4$),
so that $|\mathbf{W}(\mathbf{K})|=1,152=2^73^2$.
We need to distinguish the different cases\,:
\begin{enumerate}
\item If $\alpha=\alpha_1,\alpha_6$, then
$\mathfrak{k}^s_+\simeq \mathfrak{so}(9)$, so that
$|\mathbf{W}(\mathbf{K}_+)|=384=2^44!$ and
$\chi(M)=\chi(\tilde{M})=1,152/384=3$.
\item If $\alpha=\alpha_2,\alpha_3,\alpha_5$, then $\mathfrak{k}^s_+
=\mathfrak{l}=\mathfrak{su}(3)$. Hence 
$|\mathbf{W}(\mathbf{K}_+)|=6=3!$ and
$\chi(M)=\chi(\tilde{M})=1,152/6=192$.
\item If $\alpha=\alpha_4$, then $\mathfrak{k}^s_+
=\mathfrak{su}(2)\oplus\mathfrak{su}(2)\oplus\mathfrak{su}(2)$. Hence 
$|\mathbf{W}(\mathbf{K}_+)|=8=2^3$ and
$\chi(M)=\chi(\tilde{M})=1,152/8=144$.
\end{enumerate}
This completes the proof of the Theorem. \end{proof}


The computation of the Euler characteristic in the general case
reduces to the previous Theorem and to the well known formula 
$\chi(M)=\chi(M')\cdot\chi(M'')$, that is valid
for
a smooth fiber bundle $M\to{M}'$ with typical fiber $M''$.\par
In particular we obtain\,:

\begin{thm} Let $M$ be the minimal orbit associated to the 
effective pair
$(\mathfrak{g},\mathfrak{q})$ of the real semisimple Lie algebra
$\mathfrak{g}$ and the complex parabolic subalgebra $\mathfrak{q}$
of its complexification $\Hat{\mathfrak{g}}$. Let
$\mathfrak{g}=\bigoplus_{i=1}^m{\mathfrak{g}}_i$ be the decomposition
of $\mathfrak{g}$ into the direct sum of its simple ideals.
For each $i=1,\hdots,m$ consider the pair $(\mathfrak{g}_i,
\mathfrak{q}_i)$, for 
$\mathfrak{q}_i=\mathfrak{q}\cap\Hat{\mathfrak{g}}_i$.
Then the Euler characteristic $\chi(M)$ of $M$ is always nonnegative and
is positive if and only if
each $\g_i$
is one of the following\,: 
\begin{enumerate}
\item
of the complex type; 
\item
compact;
\item of the real types $\mathrm{A\,I\!{I}},\;\mathrm{D\,I\!{I}},\;
\mathrm{E\,I\!{V}}$;
\item of the real type $\mathrm{A\,I}$, with
$\mathfrak{g}_i\simeq\mathfrak{sl}(n,\mathbb{R})$ and
$\Phi_i=\{\alpha_{j_1},\hdots,\alpha_{j_r}\}$ 
for a sequence of integers $\{j_h\}_{0\leq h\leq r+1}$ with\,:\par\noindent
$0=j_0<j_1<\cdots<j_r<j_{r+1}=n\quad\text{and}\quad \sum_{h=0}^r{\left[
\frac{j_{h+1}-j_h}{2}\right]}=\left[\frac{n}{2}\right]\,;$
\item 
of the real type $\mathrm{D\,I}$, with
$\mathfrak{g}_i\simeq\mathfrak{so}(p,2n-p)$ with $p$ odd,
$3\leq p\leq n$ and $\Phi_i=\{\alpha_1\}$;
\item 
of the real type $\mathrm{E\,I}$,
with 
$\Phi_i\subset\{\alpha_1,\alpha_6\}$.
\end{enumerate}
\end{thm}
\begin{proof} We recall that $\chi(M)\geq 0$ by \cite{Wan1949},
because $M$ is the homogeneous space of a compact group.\par
With the notation of
Proposition \ref{prodemme}, 
$\chi(M)=\chi(M_1)\cdots\chi(M_m)$,
where $M_i$ is a minimal orbit associated to the pair $(\mathfrak{g}_i,
\mathfrak{q}_i)$.
Therefore it suffices to prove the Theorem under the
additional assumption that $\mathfrak{g}$ is simple.
\par
Let 
$\mathfrak{q}=\mathfrak{q}_{\Phi}$
for a set $\Phi$ of simple roots contained in a basis $\mathcal{B}$,
that corresponds to the nodes
of the Satake diagram $\mathcal{S}$ of $\mathfrak{g}$. 
\par
If $\alpha\in\Phi$ and
$M'$ is the minimal orbit associated to 
$(\mathfrak{g},\mathfrak{q}_{\{\alpha\}})$, 
then we have a $\mathbf{G}$-equivariant fibration $M\to{M}'$, say
with fiber $M''$, and $\chi(M)=\chi(M')\cdot\chi(M'')$. The condition
$\chi(M')>0$ is then necessary in order that $\chi(M)\neq 0$. \par
Thus, by Theorem \ref{singolo}, the conditions of the Theorem are
necessary. \par
Since $\nu(\mathcal{S})=0$
when $\mathfrak{g}$ is either of the complex type, or compact,
or of one of the real types $\mathrm{A\,I\! I},\; \mathrm{D\,I\! I},\;
\mathrm{E\,I\! V}$, all $\emptyset\neq\Phi\subset\mathcal{B}$ lead
in these cases to $\chi(M)>0$. Also the case $(5)$ is clear, because, by
Theorem \ref{singolo}, in that case
we may have $\chi(M)>0$ only with $\Phi=\{\alpha_1\}$.
\par 
Thus we only need to consider the cases $(4)$ and $(6)$. \par
\smallskip
$(4)$ \quad When $\mathfrak{g}\simeq\mathfrak{sl}(n,\mathbb{R})$ and
$\Phi=\{\alpha_{j_1},\hdots,\alpha_{j_r}\}$,
the Levi subalgebra of $\mathfrak{g}_+$ is 
$\mathfrak{l}\simeq\bigoplus_{h=0}^r{\mathfrak{sl}(j_{h+1}-j_h,\mathbb{R})}$.
Hence
$\nu''=\sum_{h=0}^r{\left[\frac{j_{h+1}-j_h}{2}\right]}$ and
the condition $\nu''=\nu=\left[\frac{n}{2}\right]$ is necessary
and sufficient for having $\chi(M)>0$. \par
Since $\mathfrak{k}^s_+\simeq\bigoplus_{h=0}^r{\mathfrak{so}(j_{h+1}-j_h)}$,
we obtain  
$\chi(M)=\frac{[n/2]!}{\prod_{h=0}^r[(j_{h+1}-j_h)/2]!}
\,.$
\par
\smallskip
$(6)$\quad By Theorem \ref{singolo}, it only remains to consider the
case where $\Phi=\{\alpha_1,\alpha_6\}$. Let $M'$ be the minimal orbit
corresponding to $(\mathfrak{e}_{\mathrm{I}},\mathfrak{q}_{\{\alpha_6\}})$
and $M''$ the fiber of 
the fibration $M\to{M}'$. By Proposition \ref{fibrazione},
$M''$ is the minimal orbit associated to
$(\mathfrak{so}(5,5),\mathfrak{q}_{\{\alpha_1\}})$. We know from
Theorem \ref{singolo} that $\chi(M')=3$ and $\chi(M'')=2$. 
Thus $\chi(M)=\chi(M')\cdot\chi(M'')=6$.
\end{proof}
\section{Some examples}\label{esempi}
\begin{exam}
The method outlined above can also be applied in
the classical cases.
Let for instance
$\mathfrak{g}=\mathfrak{so}(2n)$, with $n\geq 3$.
We can assume that $\mathfrak{q}=\mathfrak{q}_{\Phi}$ with
$\Phi=\{\alpha_{j_1},\hdots,\alpha_{j_r}\}$, 
for a sequence of integers $0=j_0<j_1<\cdots<j_r\leq{j}_{r+1}={n}$,
with 
$j_r\neq n-1$.
Then we obtain\,:
$\mathfrak{l}=\mathfrak{k}^s_+=\mathfrak{so}(2(n-j_r))\oplus
\bigoplus_{h=0}^{r-1}\mathfrak{su}(j_{h+1}-
j_h)$.   
Hence, for the corresponding $M=M_{j_1,\hdots,j_r}^{\mathrm{D}_n}$ we obtain
\begin{equation*}
\chi(M_{j_1,\hdots,j_r}^{\mathrm{D}_n})=\begin{cases}\displaystyle
\frac{2^{j_r}n!}{\prod_{h=0}^{r}(j_{h+1}-j_h)!} &\text{if}\quad
j_r\leq n-2\,,\\
\\
\displaystyle
\frac{2^{n-1}n!}{\prod_{h=0}^{r-1}(j_{h+1}-j_h)!} &\text{if}\quad
j_r=n\,.
\end{cases}
\end{equation*}
\end{exam}
\begin{exam}
Let us turn now to the case $\mathrm{D\,I\! I}$.
Let $\mathfrak{g}\simeq\mathfrak{so}(1,2n-1)$, 
with $n\geq 4$, and $\mathfrak{q}=
\mathfrak{q}_{\Phi}$ with $\Phi=\{\alpha_{j_1},\hdots,\alpha_{j_r}\}$,
where again we assume that 
$0=j_0<j_1<\cdots<j_r\leq{j}_{r+1}={n}$,
and
$j_r\neq n-1$.
We note that, by Lemma \ref{lemdiffeom}, 
$M=M^{\mathrm{[D\,I\! I]_n}}_{j_1,\hdots,j_r}$ is diffeomorphic to
the minimal orbit associated to $(\mathfrak{g},\mathfrak{q}_{\Phi'})$
where $\Phi'=\Phi\cup\{\alpha_1\}$. Thus we can as well assume that
$\alpha_1\in\Phi$, i.e. that $j_1=1$. Since we know that for the
minimal orbit $M'$ associated to $(\mathfrak{g},\mathfrak{q}_{\{\alpha_1\}})$
we have $\chi(M')=2$, we can apply Proposition \ref{fibrazione}
to the $\mathbf{G}$ equivariant fibration $M\to{M}'$. Since
the fiber $M''$ is the complex flag manifold
$M^{\mathrm{D}_{n-1}}_{j_2-1,\hdots,j_r-1}$, 
we conclude that\,:
\begin{equation*}
\chi(M^{\mathrm{[D\,I\! I]_n}}_{1,j_2,\hdots,j_r})
=2\cdot\chi(M^{\mathrm{D}_{n-1}}_{j_2-1,\hdots,j_r-1})\,.
\end{equation*}
\end{exam}
\begin{exam}
Assume that $\mathfrak{g}\simeq\mathfrak{sl}(n,\mathbb{H})$
is of the real type $[\mathrm{A\,I\! I}]_{2n-1}$ and that $\mathfrak{q}=
\mathfrak{q}_{\Phi}$ with
$\Phi=\{\alpha_{2j_1-1},\hdots,\alpha_{2j_r-1}\}$, for a sequence of
integers satisfying
$0=j_0< j_1<\cdots<j_r<j_{r+1}=n+1$. Consider the minimal orbit
$M^{[\mathrm{A\,I\! I}]_{2n-1}}_{{2j_1-1},\hdots,{2j_r-1}}$.
Since $\bar\alpha_{2h}=\alpha_{2h-1}+\alpha_{2h}+\alpha_{2h+1}$ for
$1\leq h\leq n-1$, we obtain that the Levi subalgebra of $\mathfrak{g}_+$
is
$\mathfrak{l}=\bigoplus_{h=0}^r{\mathfrak{sl}(j_{h+1}-j_h-1,\mathbb{H})}$.
Hence\,:
\begin{equation*}\begin{aligned}
\chi(M^{[\mathrm{A\,I\! I}]_{2n-1}}_{2j_1-1,\hdots,2j_r-1})&=
\frac{2^nn!}{\displaystyle
\prod_{h=0}^r\left(2^{j_{h+1}-j_h-1}(j_{h+1}-j_h-1)!\right)}\\
&=\frac{2^rn!}{\displaystyle\prod_{h=0}^r(j_{h+1}-j_h-1)!}\,.
\end{aligned}
\end{equation*}
\end{exam}
\begin{exam}
Consider the case where 
$\mathfrak{g}=\mathfrak{e}_{\mathrm{I\! V}}$. We have already discussed the
case where $\Phi\subset\{\alpha_1,\alpha_6\}$. Assume therefore
that $\Phi\cap\{\alpha_2,\alpha_3,\alpha_4,\alpha_5\}\neq\emptyset$.
We observe that, by Lemma \ref{lemdiffeom}, the minimal orbit
associated to $(\mathfrak{e}_{\mathrm{I\! V}},\mathfrak{q}_{\Phi})$
is diffeomorphic to the minimal orbit associated to
$(\mathfrak{e}_{\mathrm{I\! V}},\mathfrak{q}_{\Phi\cup\{\alpha_1,\alpha_6\}})$.
Thus we can proceed as in the discussion of the case
$[\mathrm{D\, I\! I}]$. Indeed we can assume that
$\Phi=\{\alpha_1,\alpha_{j_1},\hdots,\alpha_{j_r},
\alpha_6\}$ with $r\geq 1$. By considering the 
$\mathbf{G}$-equivariant fibration over $M'=M^{\mathrm{E\,I\! V}}_{1,6}$,
associated to 
$(\mathfrak{e}_{\mathrm{I\! V}},\mathfrak{q}_{\{\alpha_1,\alpha_6\}})$,
we obtain by Proposition \ref{fibrazione} that the fiber
is $M''=M^{\mathrm{D}_4}_{j_1-1,\hdots,j_r-1}$. Hence\,
since\,:
\begin{equation*}   
\chi(M^{\mathrm{E\,I\! V}}_{1,6})=\chi(M^{\mathrm{E\,I\! V}}_{6})\cdot
\chi(M^{[\mathrm{D\,I\! I}]_5}_{1})=3\cdot 2=6\,,\end{equation*}
we obtain\,:
\begin{equation*}
\chi(M^{\mathrm{E\,I\! V}}_{1,j_1,\hdots,j_r,6})=6\cdot
\chi(M^{\mathrm{D}_4}_{j_1-1,\hdots,j_r-1})\,.
\end{equation*}
\end{exam}

\section{Appendix}\label{appendice}
In the following table we give, for each noncompact simple Lie algebra
of the real type, a linear representation $\mathfrak{g}$, its maximal
compact subalgebra $\mathfrak{k}$, the order of the Weyl group of the
maximal compact subgroup $\mathbf{K}$ of a connected Lie group with
Lie algebra $\mathfrak{g}$, the number $\nu$ of the elements of a
maximal system of strongly orthogonal real roots of $\mathcal{R}$,
the dimension $\ell$ of a Cartan subalgebra of $\mathfrak{g}$.
The numbers $\nu$ are essentially computed in \cite{Su59}.
However, since the computation there is rather implicit, we also give
an explicit list of maximal systems of strongly orthogonal roots
for each listed case.

$$\scriptsize
\begin{array}{|c|c|c|c|c|c|} \hline
&&&&&\\[-6pt]
\text{type}&\mathfrak{g}&\mathfrak{k}&|\mathbf{W}(\mathbf{K})|&\nu&\ell\\
&&&&&\\[-6pt]
\hline
&&&&&\\[-6pt]
{\mathrm{A\,I}}&\mathfrak{sl}(n,\mathbb{R})&\mathfrak{so}(n)&
2^{\left[\frac{n+1}{2}\right]-1}\cdot 
\left[\frac{n}{2}\right]!&\left[\frac{n}{2}\right]& n-1\\
&&&&&\\[-6pt]
\hline
&&&&&\\[-6pt]
{\mathrm{A\,II}} &\mathfrak{sl}(n,\mathbb{H})&\mathfrak{sp}(n)&2^n\cdot n! & 0
&2n-1\\
&&&&&\\[-6pt]
\hline
&&&&&\\[-6pt]
{\mathrm{A\,III}}&
\begin{gathered}
\mathfrak{su}(p,q)\\
2\leq p\leq q
\end{gathered}&
\mathfrak{s}(\mathfrak{u}(p)
\oplus\mathfrak{u}(q))&
p!\cdot q! &p &p+q-1\\ 
&&&&&\\[-6pt]
\hline
&&&&&\\[-6pt]
{\mathrm{A\,IV}}&\begin{gathered}
\mathfrak{su}(1,q)\\
q\geq 1
\end{gathered}&\mathfrak{u}(q)&
q!&1 &q\\[3pt]
\hline
\end{array}
$$
$$\scriptsize
\begin{array}{|c|c|c|c|c|c|} \hline
&&&&&\\[-6pt]
\text{type}&\mathfrak{g}&\mathfrak{k}&|\mathbf{W}(\mathbf{K})|&\nu&\ell\\
&&&&&\\[-6pt] \hline
&&&&&\\[-6pt]
{\mathrm{B\,I}}&\begin{gathered}
\mathfrak{so}(p,2n+1-p)\\
2\leq{p}\leq{n}
\end{gathered}&\mathfrak{so}(p)\oplus\mathfrak{so}(2n+1-p)&
2^{n-1}\left[\frac p 2\right]!
\left[\frac{2n+1-p}{2}\right]!&p &n\\ 
&&&&&\\[-6pt]
\hline
&&&&&\\[-6pt]
{\mathrm{B\,II}}&\begin{gathered}
\mathfrak{so}(1,2n)\\
n\geq{1}
\end{gathered}&\mathfrak{so}(2n)&
2^{n-1} n!
&1 &n
\\ &&&&&\\[-6pt] \hline 
&&&&&\\[-6pt]
{\mathrm{C\,I}}&\mathfrak{sp}(2n,\mathbb{R})&\mathfrak{u}(n)&
n!
&n &n\\ &&&&&\\[-6pt] \hline &&&&&\\[-6pt]
{\mathrm{C\,II}}&
\begin{gathered}
\mathfrak{sp}(p,q)\\
0<p\leq{q}
\end{gathered}&\mathfrak{sp}(p)\oplus\mathfrak{sp}(q)&
2^{p+q} p!\cdot q!
&p &p+q\\ &&&&&\\[-6pt] \hline &&&&&\\[-6pt]
{\mathrm{D\,I}}&
\begin{gathered}
\mathfrak{so}(p,2n-p)\\
n\geq 4\\
2\leq{p}\leq{n}
\end{gathered}&\mathfrak{so}(p)\oplus\mathfrak{so}(2n-p)&
\begin{gathered}
2^{n-p+2\left[\frac{p-1}{2}\right]}\;\times\\
\qquad \times\;{\begin{smallmatrix}
\left[\frac{p}{2}\right]!\left[\frac{2n-p}{2}\right]!
\end{smallmatrix}}
\end{gathered}
&2\left[\frac{p}{2}\right] &n
\\ &&&&&\\[-6pt] \hline &&&&&\\[-6pt]
{\mathrm{D\,II}}&
\begin{gathered}
\mathfrak{so}(1,2n-1)\\
n\geq 4
\end{gathered}
&\mathfrak{so}(2n-1)&
2^{n-1}(n-1)!
&0&n
\\ &&&&&\\[-6pt] \hline &&&&&\\[-6pt]
{\mathrm{D\,III}}&
\begin{gathered}
\mathfrak{so}^*(2n)\\
n\geq 2
\end{gathered}
&\mathfrak{u}(n)&
n!
&\left[\frac{n}{2}\right]&n
\\ &&&&&\\[-6pt] \hline &&&&&\\[-6pt]
{\mathrm{E\,I}}&\mathfrak{e}_{\mathrm{I}}&\mathfrak{sp}(4)&
384
&4&6
\\ &&&&&\\[-6pt] \hline 
&&&&&\\[-6pt]
{\mathrm{E\,II}}&\mathfrak{e}_{\mathrm{I\! I}}
&\mathfrak{su}(2)\oplus\mathfrak{su}(6)&
1,\! 440
&4&6
\\ &&&&&\\[-6pt] \hline &&&&&\\[-6pt]
{\mathrm{E\,III}}&\mathfrak{e}_{\mathrm{I\! I\! I}}
&\mathfrak{so}(10)\oplus\mathbb{R}&
1,\!920
&2&6
\\ &&&&&\\[-6pt] \hline
&&&&&\\[-6pt]
{\mathrm{E\,IV}}&\mathfrak{e}_{\mathrm{I\! V}}&\mathfrak{f}_4&
1,\!152
&0&6
\\ &&&&&\\[-6pt] 
\hline &&&&&\\[-6pt]
{\mathrm{E\,V}}&\mathfrak{e}_{\mathrm{V}}&\mathfrak{su}(8)&
40,\!320
&7&7
\\ &&&&&\\[-6pt] \hline &&&&&\\[-6pt]
{\mathrm{E\,VI}}&\mathfrak{e}_{\mathrm{V\! I}}
&\mathfrak{su}(2)\oplus\mathfrak{so}(12)&
46,\!080
&4&7
\\ &&&&&\\[-6pt] \hline &&&&&\\[-6pt]
{\mathrm{E\,VII}}&\mathfrak{e}_{\mathrm{V\! I\! I}}
&\mathfrak{e}_6\oplus\mathbb{R}&
51,\!840
&3&7
\\ &&&&&\\[-6pt] \hline &&&&&\\[-6pt]
{\mathrm{E\,VIII}}&\mathfrak{e}_{\mathrm{V\! I\! I\! I}}
&\mathfrak{so}(16)&
5,\!160,\!960
&8&8
\\ &&&&&\\[-6pt] \hline &&&&&\\[-6pt]
{\mathrm{E\,IX}}&\mathfrak{e}_{\mathrm{I\! X}}
&\mathfrak{su}(2)\oplus\mathfrak{e}_7&
5,\!806,\!080
&4&8
\\ &&&&&\\[-6pt] \hline &&&&&\\[-6pt]
{\mathrm{F\,I}}&\mathfrak{f}_{\mathrm{I}}&\mathfrak{su}(2)\oplus
\mathfrak{sp}(3)&
96
&4&4
\\ &&&&&\\[-6pt] \hline &&&&&\\[-6pt]
{\mathrm{F\,II}}&\mathfrak{f}_{\mathrm{I\! I}}&\mathfrak{so}(9)&
384
&1&4
\\ &&&&&\\[-6pt] \hline &&&&&\\[-6pt]
{\mathrm{G\,I}}&\mathfrak{g}_{\mathrm{I}}
&\mathfrak{so}(3)\oplus\mathfrak{so}(3)&
16
&2&2
\\[3pt]
\hline 
\end{array}
$$
\par\bigskip
We denote by $\Gamma$ a maximal system of strongly orthogonal real
roots in $\mathcal{R}$. For each case of the list we describe below
an explicit $\Gamma$.
\begin{align*}
\intertext{$[\mathrm{A}_{\ell}]$.
\quad $\mathcal{R}=\{\pm(e_i-e_j)\,|\, 1\leq i<j\leq n\}$.}
&\mathrm{A\,I} &\nu&=\left[{n}/{2}\right]&\Gamma&=\{e_i-e_{n+1-i}\,|\,
1\leq i\leq\nu\}\\
&\mathrm{A\,II}&\nu&=0&\Gamma&=\emptyset\\
&\mathrm{A\,III}&\nu&=p&\Gamma&=\{e_i-e_{n+1-i}\,|\,1\leq i\leq\nu\}\\
&\mathrm{A\,IV}&\nu&=1&\Gamma&=\{e_1-e_n\}
\intertext{$[\mathrm{B}_{\ell}]$.
\quad $\mathcal{R}=\{\pm e_i\,|\, 1\leq i\leq\ell\}\cup\{ 
\pm e_i\pm e_j\,|\, 1\leq i<j\leq\ell\}$}
&\mathrm{B\,I}&\nu&=p\;\mathrm{even}&
\Gamma&=\left\{e_{2j-1}\pm{e}_{2j}\,|\, 1\leq j\leq {p}/{2}\right\}\\
&\mathrm{B\,I}&\nu&=p\;\mathrm{odd}
&\Gamma&=\left\{e_{2j-1}\pm{e}_{2j}\,|\, 1\leq j\leq {(p-1)}/{2}
\right\}\cup\{e_p\}\\
&\mathrm{B\,II}&\nu&=1&\Gamma&=\{e_1\}
\intertext{$[\mathrm{C}_{\ell}]$.
\quad $\mathcal{R}=\{\pm{2e_i}\,|\, 1\leq i\leq\ell\}\cup
\{\pm e_i\pm{e}_j\,|\, 1\leq i<j\leq\ell\}$.}
&\mathrm{C\,I}&\nu&=\ell&\Gamma&=\{2e_i\,|\, 1\leq i\leq\ell\}\\
&\mathrm{C\,II}&\nu&=p&\Gamma&=\{e_{2i-1}+e_{2i}\,|\, 1\leq i\leq p\}\\
\intertext{$[\mathrm{D}_{\ell}]$.
\quad $\mathcal{R}=\{\pm{e}_i\pm{e}_j\,|\, 1\leq i<j\leq\ell\}$.}
&\mathrm{D\,I}&\nu&=2\left[{p}/{2}\right]&
\Gamma&=\{e_{2h-1}\pm{e}_{2h}\,|\, 2h\leq p\}\\
&\mathrm{D\,II}&\nu&=0&\Gamma&=\emptyset\\
&\mathrm{D\,III}&\nu&=[{\ell}/{2}]&\Gamma&=
\{e_{2h-1}+e_{2h}\,|\, 1\leq h\leq [\ell/2]\}\,.
\end{align*}
\subsection*{$[{{\mathrm{E}}_6,\; {\mathrm{E}}_7,\; {\mathrm{E}}_8]}$}
\quad
Following \cite{Bou68} we set\,: 
\par
\begin{equation*}
\mathcal{R}(\mathfrak{e}_8)=\{\pm{e}_i\pm{e}_j)\,|\, 1\leq i<j\leq 8\}
\cup
\left.\left\{\frac{1}{2}\sum_{i=1}^8(-1)^{k_i}e_i\,\right|\, k_i\in
\mathbb{Z}\,,\;
\sum_{i=1}^8k_i\in 2\mathbb{Z}\right\},\end{equation*}\par
$\mathcal{R}(\mathfrak{e}_7)=\mathcal{R}(\mathfrak{e}_8)\cap\{e_7+e_8\}^\perp$,
\par\smallskip
$\mathcal{R}(\mathfrak{e}_6)=\mathcal{R}(\mathfrak{e}_8)\cap\{e_6+e_8,
e_7+e_8\}^\perp$,\par\smallskip
so that in particular $\mathcal{R}(\mathfrak{e}_6)\subset
\mathcal{R}(\mathfrak{e}_7)\subset\mathcal{R}(\mathfrak{e}_8)$.\par
Likewise, the basis of simple roots can be considered as included one
into the other: $\mathcal{B}(\mathfrak{e}_6)=\{\alpha_1,\hdots,\alpha_6\}
\subset\mathcal{B}(\mathfrak{e}_7)=\{\alpha_1,\hdots,\alpha_7\}
\subset\mathcal{B}(\mathfrak{e}_8)=\{\alpha_1,\hdots,\alpha_8\}$
for $\alpha_1=\frac{1}{2}(e_1-e_2-e_3-e_4-e_5-e_6-e_7+e_8)$,
$\alpha_2=e_1+e_2$, $\alpha_3=e_2-e_1$, $\alpha_4=e_3-e_2$,
$\alpha_5=e_4-e_3$, $\alpha_6=e_5-e_4$, $\alpha_7=e_6-e_5$,
$\alpha_8=e_7-e_6$.\par
Set\,:
$$\beta_i=\left\{\begin{aligned}
e_{i}-e_{i+1}&\quad{\mathrm{for\; odd}}\quad i\\
e_{i-1}+e_{i}&\quad{\mathrm{for\; even}}\quad i\,,\end{aligned}
\right.
$$
so that $\beta_i\in\mathcal{R}(\mathfrak{e}_6)$ for $i\leq 4$,
$\beta_i\in\mathcal{R}(\mathfrak{e}_7)$ for $i\leq 7$,
$\beta_i\in\mathcal{R}(\mathfrak{e}_8)$ for $i\leq 8$.
Note that $\{\beta_i\,|\, 1\leq i\leq 8\}$ 
is a system of eight strongly orthogonal roots
in $\mathcal{R}(\mathfrak{e}_8)$. 
\par
Since the conjugation in the non-split forms are better expressed
in terms of the simple
roots $\alpha_i$, we also found convenient,
to describe the maximal sets $\Gamma$,  to introduce 
other roots $\gamma_i$,
defined as linear combinations of the simple roots
$\alpha_i$\,:
\par
$\gamma_1=\alpha_1+2\alpha_2+
2\alpha_3+3\alpha_4+2\alpha_5+\alpha_6={\begin{smallmatrix}
1&2&3&2&1\\
&&2
\end{smallmatrix}}\in\mathcal{R}(\mathfrak{e}_6)
$\par
$\gamma_2=\alpha_1+
\alpha_3+\alpha_4+\alpha_5+\alpha_6={\begin{smallmatrix}
1&1&1&1&1\\
&&0
\end{smallmatrix}}\in\mathcal{R}(\mathfrak{e}_6)
$\par
$\gamma_3=
\alpha_3+\alpha_4+\alpha_5={\begin{smallmatrix}
0&1&1&1&0\\
&&0
\end{smallmatrix}}\in\mathcal{R}(\mathfrak{e}_6)
$\par
$\gamma_4=\alpha_1+\alpha_2+2\alpha_3+2\alpha_4+\alpha_5=\begin{smallmatrix}
1&2&2&1&0&0\\
&&1
\end{smallmatrix}\in\mathcal{R}(\mathfrak{e}_7)$\par
$\gamma_5=\alpha_1+\alpha_2+2\alpha_3+2\alpha_4+2\alpha_5+2\alpha_6+\alpha_7=
\begin{smallmatrix}
1&2&2&2&2&1\\
&&1\end{smallmatrix}\in\mathcal{R}(\mathfrak{e}_7)$\par
$\gamma_6=\alpha_1+2\alpha_2+2\alpha_3+4\alpha_4+3\alpha_5+2\alpha_6+\alpha_7=
\begin{smallmatrix}
1&2&4&3&2&1\\
&&2
\end{smallmatrix}
\in\mathcal{R}(\mathfrak{e}_7)$\par
$\gamma_7=\alpha_2+\alpha_3+2\alpha_4+2\alpha_5+2\alpha_6+\alpha_7=
\begin{smallmatrix}
0&1&2&2&2&1\\
&&1
\end{smallmatrix}
\in\mathcal{R}(\mathfrak{e}_7)$\par
$\gamma_8=2\alpha_1+2\alpha_2+3\alpha_3+4\alpha_4+3\alpha_5+2\alpha_6+\alpha_7=
\begin{smallmatrix}
2&3&4&3&2&1\\
&&2
\end{smallmatrix}
\in\mathcal{R}(\mathfrak{e}_7)$\par
$\gamma_9=2\alpha_1+3\alpha_2+4\alpha_3+6\alpha_4+5\alpha_5+4\alpha_6+3\alpha_7
+2\alpha_8=
\begin{smallmatrix}
2&4&6&5&4&3&2\\
&&3
\end{smallmatrix}
\in\mathcal{R}(\mathfrak{e}_8)$\par
We can describe a system $\Gamma$ of strongly orthogonal roots for the
real simple Lie algebra of the exceptional types $\mathrm{E}_6,\;
\mathrm{E}_7,\;\mathrm{E}_8$ by\,:
\begin{align*}
&\mathrm{E\,I}&\nu&=4&\Gamma&=\{\beta_1,\beta_2,\beta_3,\beta_4\}\\
&\mathrm{E\,II}&\nu&=4&\Gamma&=\{\alpha_4,\gamma_1,\gamma_2,\gamma_3\}\\
&\mathrm{E\,III}&\nu&=2&\Gamma&=\{\gamma_1,\gamma_2\}\\
&\mathrm{E\,IV}&\nu&=0&\Gamma&=\emptyset\\
&\mathrm{E\,V}&\nu&=7&\Gamma&=
\{\beta_1,\beta_2,\beta_3,\beta_4,\beta_5,\beta_6,
\beta_7\}\\
&\mathrm{E\,VI}&\nu&=4&\Gamma&=\{\alpha_1,\gamma_4,\gamma_5,\gamma_6\}\\
&\mathrm{E\,VII}&\nu&=3&\Gamma&=\{\alpha_7,\gamma_7,\gamma_8\}\\
&\mathrm{E\,VIII}&\nu&=8&\Gamma&=\{\beta_1,\beta_2,\beta_3,
\beta_4,\beta_5,\beta_6
,\beta_7,\beta_8\}\\
&\mathrm{E\,IX}&\nu&=4&\Gamma&=\{\alpha_7,\gamma_7,\gamma_8,\gamma_9\}\\
\end{align*}
\subsection*{$[\mathrm F_4]$} 
\quad We have\,:
\begin{equation*}
\mathcal{R}=\{\pm{e}_i\,|\, 1\leq i\leq 4\} \cup
\{\pm e_i\pm e_j\,|\, 1\leq i<j\leq 4\}
\cup
\left\{\frac{\pm e_1\pm e_2\pm e_3\pm e_4}{2}\right\}.
\end{equation*}
 Then we have\,:
\begin{align*}
&{\mathrm{ F\,I}}&\nu&=4&\Gamma&=\{e_1\pm{e}_2,e_3\pm{e}_4\}\\
&{\mathrm{ F\, II}}&\nu&=1&\Gamma&=\{\alpha_1+2\alpha_2+3\alpha_3+2\alpha_4\}=
\{e_1\}
\intertext{$[{\mathrm G_2}]$.
\quad We have\,:
\begin{equation*}
\mathcal{R}=\{\pm(e_i-e_j)\,|\, 1\leq i<j\leq 3\}
\cup
\left\{\pm(2e_i-e_j-e_k)\,|\, \{i,j,k\}=\{1,2,3\}\right\}.
\end{equation*}}
&{\mathrm{ G\,I}}&\nu&=2&\Gamma&=\{e_1-{e}_2,2e_3-e_1-e_2\}\, .
\end{align*}

\providecommand{\bysame}{\leavevmode\hbox to3em{\hrulefill}\thinspace}
\providecommand{\href}[2]{#2}

\end{document}